\documentclass[12pt]{amsart}
\usepackage{graphicx} 
\usepackage{amssymb} 

\usepackage[utf8]{inputenc}
\usepackage{enumerate}          
\usepackage{amsmath, amsthm}    
\usepackage{wrapfig} 
\graphicspath{ {./images/} }
\usepackage{rotating}
\usepackage{bm}
\usepackage{verbatim}
\usepackage[dvipsnames]{xcolor} 
\usepackage{overpic} 
\usepackage{hhline}
\usepackage{ulem}
\usepackage{tikz-cd}
\usetikzlibrary{decorations.pathreplacing}
\usepackage{hyperref}           

\definecolor{insgreen}{rgb}{0,0.45,0}
\definecolor{delred}{rgb}{0.80,0,0}

\newtheorem{theorem}{Theorem}[subsection]
\newtheorem{corollary}[theorem]{Corollary}
\newtheorem{lemma}[theorem]{Lemma}
\newtheorem{proposition}[theorem]{Proposition}
\newtheorem{thm}{Theorem}

\theoremstyle{definition}
\newtheorem{definition}[theorem]{Definition}
\theoremstyle{definition}
\newtheorem{example}[theorem]{Example}
\theoremstyle{definition}
\newtheorem{remark}[theorem]{Remark}
\theoremstyle{definition}

\makeatletter
\renewcommand\@seccntformat[1]{%
	\protect\textup{%
		\expandafter\ifx\csname the#1\endcsname\thesubsection
		\protect\bfseries
		\else
		\protect\@secnumfont
		\fi
		\csname the#1\endcsname
		\protect\@secnumpunct
	}%
}
\makeatother

\newcommand{\Z}{\mathbb{Z}}

\newcommand{\R}{\mathbb{R}}
\newcommand{\C}{\mathbb{C}}

\newcommand{\inv}{^{-1}}

\newcommand{\RP}{{\mathbb{RP}}{}}
\DeclareMathOperator{\id}{id}
\DeclareMathOperator{\PD}{PD}
\DeclareMathOperator{\Int}{Int}
\let\Bd\partial

\makeatletter

\newcommand{\eusm@selectfont}[1]{%
	\ifdim#1>10pt
	\font\eusm@tmp=eusm10 at #1\relax
	\else\ifdim#1>9pt
	\font\eusm@tmp=eusm10\relax
	\else\ifdim#1>8pt
	\font\eusm@tmp=eusm9\relax
	\else\ifdim#1>7pt
	\font\eusm@tmp=eusm8\relax
	\else\ifdim#1>6pt
	\font\eusm@tmp=eusm7\relax
	\else\ifdim#1>5pt
	\font\eusm@tmp=eusm6\relax
	\else
	\font\eusm@tmp=eusm5\relax
	\fi\fi\fi\fi\fi\fi
}

\newcommand{\eusm@makebox}[2]{%
	\begingroup
	\dimen@=\f@size pt\relax   
	\dimen@=#1\dimen@          
	\eusm@selectfont{\the\dimen@}%
	\hbox{\eusm@tmp #2}%
	\endgroup
}

\newcommand{\eusmletter}[1]{%
	\ifmmode
	\mathchoice
	{\mathord{\eusm@makebox{1}{#1}}}
	{\mathord{\eusm@makebox{1}{#1}}}
	{\mathord{\eusm@makebox{0.7}{#1}}}
	{\mathord{\eusm@makebox{0.5}{#1}}}
	\else
	\eusm@makebox{1}{#1}%
	\fi
}

\newcommand{\B}{\eusmletter{B}}
\renewcommand{\L}{\eusmletter{L}} 
\newcommand{\M}{\eusmletter{M}}
\newcommand{\T}{\eusmletter{T}}

\makeatother

\usepackage[left=1in, right=1in, top=.75in, bottom=.75in]{geometry} 
\begin{document}

\title[Lifting braids over $\RP^2$]{Lifting braids over $\RP^2$ and\\ non-orientable Lefschetz fibrations}

\author{Valentina Bais}
\address{Max Planck Institute for Mathematics \\ Bonn \\ Germany.}
\email{bais@mpim-bonn.mpg.de}

\author{Porter Morgan}
\address{Department of Mathematics \\ University of Texas at Austin \\ Austin TX \\  United States.}
\email{porter.morgan@austin.utexas.edu}

\author{Daniele Zuddas}
\address{Dipartimento di Matematica, Informatica e Geoscienze\\Università degli Studi di Trieste \\Trieste\\Italy.}
\email{dzuddas@units.it}

\begin{abstract}
    We show that, given a closed connected non-orientable surface $F$ and a simple branched cover $p \colon F \rightarrow \mathbb{RP}^2$ of degree $d \geq 5$, every self-diffeomorphism of $F$ is isotopic to the lift of a self-diffeomorphism of $\mathbb{RP}^2$ with respect to $p$. As an application, we show that every Lefschetz fibration over a connected surface $B$ with $\partial B\neq \emptyset$ and closed non-orientable regular fiber can be expressed as the composition of a simple cover of $B\times \mathbb{RP}^2$ branched over a surface braided over $B$ and projection onto the $B$-factor.
\end{abstract}

\maketitle

\section{Introduction}

It is a classical result by Hilden \cite{Hilden} and Berstein--Edmonds \cite{BE} that, 
given a closed connected orientable surface $F$ and a simple branched cover $p \colon F \rightarrow S^2$ of degree $d \geq 3$, every self-diffeomorphism of $F$ is isotopic to the lift of a self-diffeomorphism of $S^2$ with respect to $p$. This was Hilden's main tool to show that every closed connected orientable $3$-manifold is a simple $3$-fold cover of $S^3$ branched over a link, a result that was independently proven at the same time by Hirsch \cite{Hirsch1974} and Montesinos \cite{montesinos1974representation}.

The present paper is aimed at proving the non-orientable counterpart of Hilden's result for simple branched covers of the real projective plane $\RP^2$ of degree at least five. The following is our main theorem.

\begin{thm}\label{thm:1}
    Let $F$ be a closed connected non-orientable surface and let $p \colon F \rightarrow \RP^2$ be a simple branched cover of degree $d \geq 5$. Then, every self-diffeomorphism of $F$ is isotopic to the lift of a self-diffeomorphism of $\RP^2$.
\end{thm}

In other words, Theorem \Ref{thm:1} shows that the lifting homomorphism
    \[\lambda_p \colon \L_p \longrightarrow  \M(F) \]
    is surjective.
We refer the reader to Definition \ref{def:lifthom} for a precise description of the lifting homomorphism $\lambda_p$ and of the subgroup of liftable braids $\L_p \subset \B_n(\RP^2)$.

As a consequence, we get the following result, which extends \cite[Proposition 2]{LoiPiergallini} to the non-orientable case.

\begin{thm}\label{thm:3}
Let $f \colon X \rightarrow B$ be a Lefschetz fibration whose regular fiber $F$ is a closed connected non-orientable surface. Moreover, we assume that $B$ is a connected surface with non-empty boundary. Then, for every integer $d \geq 6$ with the same parity as $\chi(F)$, there exists a simple $d$-fold cover \[p \colon  X\longrightarrow B \times \RP^2\] branched over a surface $S \subset B \times \RP^2$ braided over $B$ such that the following diagram commutes
\begin{center}
\begin{tikzcd}
	X \arrow[rd, "f"'] \arrow[r, "p"] & B \times \RP^2  \arrow[d, "\pi_{B}"] \\
                                  & B.
\end{tikzcd}
\end{center}

Here $\pi_{B} \colon B \times \RP^2 \rightarrow B$ denotes the canonical projection onto $B$.
\end{thm}

\subsubsection*{Proof strategies} The proof of Theorem \ref{thm:1} is roughly as follows. On a closed non-orientable surface $F$, Dehn twists generate an index two subgroup of the mapping class group $\M(F)$, which we denote by $\T(F)$. We first prove that, for any integer $d \geq 5$, there exists a simple $d$-fold branched cover $p \colon F \rightarrow \mathbb{RP}^2$ such that the image of the lifting homomorphism $\lambda_p$ contains the whole $\T(F)$. This is achieved by explicitly showing that a set of generators of $\T(F)$ described by Chillingworth in \cite{Chillingworth_1969} lies in the image of $\lambda_p$. We then argue that $\T(F)$ is contained in the image of $\lambda_q$ for any other simple branched cover $q$ with the same degree as $p$, since $p$ and $q$ are equivalent by a result of Gabai--Kazez \cite{GK2}. To conclude, we construct for every degree $d \geq 5$ a simple branched cover for which an element in $\M(F) \setminus \T(F)$ is in the image of the cover's lifting homomorphism. The theorem then follows from the fact that $\T(F)$ has index two in $\M(F)$ and another application of \cite{GK2}.

A summary of the proof of Theorem \ref{thm:3} is as follows. We identify the regular fiber of the Lefschetz fibration $f\colon X\to B$ with a closed non-orientable surface $F$. Using Theorem \ref{thm:1}, we show that for any $d$-fold covering of degree at least $6$, the Dehn twist along each vanishing cycle of $f$ is the lift of a half twist around an embedded arc on $\RP^2$. Using this property, we adopt methods from \cite{LoiPiergallini} and construct a branched cover $p$ of $B\times \RP^2$ whose branch set itself admits a simple branched cover of $B$. The monodromy associated to the composition of this branched cover with projection onto $B$ coincides with the monodromy of $f$, which confirms that the total space of the branched cover is $X$ and that $f$ and $\pi_B \circ p$ are equivalent as Lefschetz fibrations.

\subsubsection*{Historical note} There is an extensive history of research on simple branched covers and their relation to Lefschetz fibrations. Early studies of lifting homomorphisms associated to covering maps between surfaces include \cite{birman_hilden_1, birman_hilden_2}, which studied the hyperelliptic double branched cover. Building off of this, more general results about the lifting homomorphism of covers between orientable surfaces were given in \cite{BE, Hilden}. Theorem \ref{thm:1} extends results in this literature to the non-orientable setting.

The connection between branched covers over braided surfaces and Lefschetz fibrations is at the core of Theorem \ref{thm:3}. Factorizations of certain orientable Lefschetz fibrations through branched covers were given in \cite{LoiPiergallini}, and this factorization was further studied and applied in \cite{Apostolakis_2013, Zuddas2007RepresentingDT}. Other instances where Lefschetz fibrations have been recognized as double branched covers may be found in \cite{Fuller2000, Siebert19990}. While non-orientable Lefschetz fibrations have also attracted attention \cite{BaykurMorgan, MillerOzbagci}, Theorem \ref{thm:3} is the first instance in which they have been factorized through branched covers.

\subsubsection*{Outline} In Section \ref{sec:2} we provide background information about branched covers. Then in Section \ref{sec:3} we give an overview of the mapping class group for non-orientable surfaces. In Section \ref{sec:4} we show how to represent 2-dimensional branched covers using \textsl{projection diagrams}, and introduce the covering maps which are central to our proof of Theorem \ref{thm:1}. Then in Section \ref{sec:5} we define the braid group and the lifting homomorphism, which allow us to formulate the statement of Theorem \ref{thm:1} more precisely. In Section \ref{sec:lift_chill} we prove Theorem \ref{thm:1}. In Section \ref{sec:9} we show how to describe Dehn twists as lifts of half twists, while in Sections \ref{sec:6} and \ref{sec:7} we define braided surfaces and Lefschetz fibrations, which are necessary background to our proof of Theorem \ref{thm:3}. Finally, in Section \ref{sec:lef} we prove Theorem \ref{thm:3}.

\subsubsection*{Acknowledgments}
Valentina Bais and Daniele Zuddas acknowledge GNSAGA – Isti\-tuto Nazionale di Alta Matematica ``Francesco Se\-veri'', Italy. The authors acknowledge the use of LLMs for grammatical revisions to this paper. There was no LLM assistance with the mathematical content of the paper, and the ideas presented here belong entirely to the authors.

\section{Branched covers}\label{sec:2}
\begin{definition}[Branched cover] Let $X$ and $Y$ be smooth compact $n$-manifolds with possibly non-empty boundary. A smooth map
\[p \colon X \longrightarrow Y\] is called a $d$-fold \textsl{branched cover} if there exists a properly embedded smooth $(n-2)$-submanifold $B_p \subset Y$ such that 
\begin{itemize}
    \item[(a)] the restriction of $p$ over the complement $Y \setminus B_p$ is an ordinary $d$-fold cover;
    \item[(b)] for every $x \in p^{-1}(B_p)$, there are local charts centered at $x$ and $p(x)$ respectively in which $p$ is smoothly equivalent to the map
    \[\mathbb{C} \times \mathbb{E}^{n-2} \longrightarrow \mathbb{C} \times \mathbb{E}^{n-2}\]
    \[(z,t)\longmapsto (z^k,t)\]
    for some positive integer $k \leq d$, where $\mathbb{E}^{n-2}=\mathbb{R}^{n-2}$ if $x \in \Int X$, $p(x)\in \Int Y$ and $\mathbb{E}^{n-2}=\mathbb{R}^{n-3} \times [0,+\infty)$ if $x \in \partial X$, $p(x)\in \partial Y$.
\end{itemize} 

The integer $d$ is called the \textsl{degree} of the cover, $k$ is the local degree of $p$ at $x$, while the submanifold $B_p\subset Y$ is the \textsl{branch set} of $p$. 

Two branched covers $p_1, p_2\colon X\rightarrow Y$ are \textsl{equivalent} if there are diffeomorphisms \[h\colon X\longrightarrow X \quad \text{and} \quad k\colon Y\longrightarrow Y\] such that $k\circ p_1 = p_2\circ h$. Moreover, we say that $p_1$ and $p_2$ are \textsl{strongly equivalent} if they are equivalent as above and $k \colon Y \rightarrow Y$ is smoothly isotopic to the identity of $Y$.
\end{definition}

\begin{remark}\label{rmk:equivRP2}
Every self-diffeomorphism of $\mathbb{RP}^2$ is smoothly isotopic to the identity \cite{Chillingworth_1969}, therefore any two branched covers of $\mathbb{RP}^2$ are equivalent if and only if they are strongly equivalent.
\end{remark}

We now recall the notion of monodromy for a $d$-fold branched cover $p \colon X \rightarrow Y$.

\begin{definition}
    Let $p \colon X \rightarrow Y$ be a $d$-fold branched cover with branch set $B_p \subset Y$. The \textsl{monodromy of $p$} is the monodromy representation \[\omega_p \colon \pi_1(Y \setminus B_p) \longrightarrow S_d\] of the associated un-branched cover \[p|_{X \setminus p^{-1}(B_p)} \colon X \setminus p^{-1}(B_p) \longrightarrow Y\setminus B_p.\] 
    Here $S_d$ denotes the permutation group of $d$ elements. 
\end{definition}
\begin{remark}
    The representation $\omega_p$ depends on a choice of base point and of a numbering of the sheets, and is therefore only well defined up to conjugation in $S_d$; this ambiguity will matter in Proposition \ref{prop:liftcrit}. 
\end{remark}
\begin{remark}
Given a smooth, compact, connected $n$-manifold $Y$, a properly embedded smooth $(n-2)$-submanifold $B \subset Y$ and a representation $\omega \colon \pi_1(Y \setminus B) \rightarrow S_d$, there exists a smooth compact $n$-manifold $X$ and a $d$-fold branched cover $p\colon X\to Y$ whose branch set $B_p$ is $B$ and whose monodromy representation $\omega_p$ coincides with $\omega$. Moreover, $X$ is unique up to diffeomorphism and $p$ is unique up to strong equivalence.
\end{remark}
In this paper we mostly work with \textsl{simple} branched covers. This means that we always assume that the monodromy representation $\omega_p$ sends meridians of the branch set $B_p$ to transpositions in $S_d$, where $d$ is the degree of $p$. This is equivalent to assuming that the preimage of each element in $B_p$ consists of exactly $d-1$ points. If $Y$ is simply connected, $\omega_p$ is completely described by its values on a system of meridians for $B_p$ generating $\pi_1(Y \setminus B_p)$. The assignment of a permutation to a system of meridians of $B_p$ is usually called a \textsl{labeling} for $B_p$. When $Y$ has non-trivial fundamental group, the situation is slightly more complicated, since $\pi_1(Y \setminus B_p)$ is not generated by just the meridians of $B_p$: one must also record the monodromy of a set of generators of $\pi_1(Y)$. In what follows, we often work with branched covers over $\mathbb{RP}^2$. We now outline some basic facts about such covers. 

Let $F$ be a closed, connected surface and $p \colon F \rightarrow \mathbb{RP}^2$ a simple $d$-fold branched cover. Writing $B_p=\{b_1,\dots,b_n\}$ and letting $\ell$ denote the class of the standard $\RP^1$, we have the presentation
\[\pi_1(\RP^2 \setminus B_p) = \langle\, \mu_1, \dots, \mu_n, \ell \mid \ell^2=\mu_1\cdots \mu_n \,\rangle,\]
where $\mu_i$ is a meridian of $b_i$ in a Hurwitz system based at a chosen base point. Thus $\omega_p$ is determined by the $n+1$ permutations $\omega_p(\mu_i)$ and $\omega_p(\ell)$, subject to the single relation $\omega_p(\ell)^2=\omega_p(\mu_1)\cdots\omega_p(\mu_n)$, and $p$ is simple exactly when every $\omega_p(\mu_i)$ is a transposition. We write
\[\mathbb{RP}^2= D^2 \cup_\partial Mb\]
where the Möbius band $Mb$ is a tubular neighborhood of the standard $\mathbb{RP}^1\subset \mathbb{RP}^2$. Moreover, we choose a base point of $\pi_1(\mathbb{RP}^2\setminus B_p)$ in $D^2 \cap Mb$. Without loss of generality, we can always assume that $B_p\cap Mb=\emptyset$. In this way, we can view $p$ as the gluing of two covers
\begin{equation}\label{eq:decomposition}
    p_1 \colon F_1 \longrightarrow D^2 \quad \text{and} \quad p_2 \colon F_2 \longrightarrow Mb,
\end{equation}
where $p_1$ is a simple cover branched over $B_p$ whose monodromy is induced by $\omega_p$ and $p_2$ is a regular cover. In particular, $p_2$ is fully determined by $\omega_{p_2}([\RP^1])$, the monodromy of the core of the Möbius band. 

Going the other way, suppose we start with two covers $p_1\colon F_1\rightarrow D^2$ and $p_2\colon F_2\rightarrow Mb$, where $\partial F_1\cong \partial F_2$, and we wish to obtain a cover $p\colon F_1\cup_\partial F_2 \rightarrow \RP^2$. A necessary and sufficient condition for such a $p$ to exist is that the restrictions of $p_1$ and $p_2$ to $\partial F_1$ and $\partial F_2$ are equivalent. One can then use the diffeomorphisms given by the equivalence to perform the desired gluing.

\begin{remark}
    Since $w_1(F)=p^*(w_1(\mathbb{RP}^2))$ (see \cite[Section 1]{bais2025branched}) and $w_1(\mathbb{RP}^2)=\PD([\mathbb{RP}^1])$, $F_2\subset F$ is nothing but a tubular neighborhood of a representative of the Poincar\'e dual of $w_1(F)$, which we will often abbreviate to $\PD(w_1(F))$.
\end{remark}

\begin{remark}\label{rmk:RH}
    The Riemann--Hurwitz formula implies that \[\chi(F)=d\,\chi(\RP^2)-|B_p|=d-|B_p|,\] so that
    \[|B_p| = d-\chi(F).\]
 In particular, the number of branch points of a simple $d$-fold branched cover $p \colon F \to \RP^2$ is determined by the degree $d$ and the Euler characteristic of $F$.
\end{remark}

In this paper, we will repeatedly make use of the following fact, which is a special case of \cite[Theorem 1.1]{GK2} for simple branched covers of $\mathbb{RP}^2$ (see Remark \ref{rmk:equivRP2}).

\begin{theorem}[Gabai--Kazez \cite{GK2}]\label{thm:gabai_kazez} 
    Let $F$ be a closed connected non-orientable surface. Two simple branched covers $f,g\colon F\rightarrow \RP^2$ are equivalent if and only if they have the same degree.
\end{theorem}

\begin{remark}
    Given a simple $d$-fold branched cover $p \colon F \rightarrow \RP^2$, we have
    \begin{align*}
    \chi(F) \equiv \langle w_2(F),[F] \rangle = \langle w_1^2(F),[F] \rangle =\langle p^*(w_1^2(\RP^2)),[F]\rangle= \\ =\langle w_1^2(\RP^2),p_*[F]\rangle= d\langle w_1^2(\RP^2),[\RP^2]\rangle \equiv d \bmod{2}.
    \end{align*}
    Here we are using the fact that $p^*(w_1(\RP^2))=w_1(F)$, the Wu formula $w_1^2=w_2$ on a closed surface together with $\langle w_2(F),[F]\rangle\equiv\chi(F) \pmod 2$ and that $p$ has topological degree $d$. So, the parity of the degree of $p$ is constrained by the parity of the Euler characteristic (or, equivalently, of the non-orientable genus) of $F$.
\end{remark}

\subsection{Stabilizations of branched covers over the real projective plane}

Given a $d$-fold simple branched cover $p \colon F_1 \rightarrow F_2$ with $F_2=S^2$ (resp. $F_2=D^2$), one can \textsl{stabilize} $p$ to get a $(d+1)$-fold simple branched cover $p'\colon F_1 \rightarrow F_2$. At the level of labeled branch sets, this amounts to adding to $B_p$ two (resp. one) new branch points with labeling $(i \, d+1)$ for some $i \in \{1, \dots, d\}$. In order to increase the degree of a given cover of $\mathbb{RP}^2$, we will make use of the following, which is Proposition 4.2 in \cite{bais2025branched} restricted to dimension two.
 \begin{proposition}\label{prop:stab}
      Let $X$ be a closed surface and 
    \[p\colon X \longrightarrow \RP^{2}\]
    a $d$-fold branched cover. Then, $p$ can be stabilized to a $(d+2)$-fold branched cover 
    \[p'\colon X \longrightarrow \RP^{2},\]
    where the branch set $B_{p'}$ is the union of $B_p$ and two separated new branch points with monodromy $(i\;d+1)$ for some $i \in \{1, \dots, d\}$, and the monodromy of $\RP^1$ is given by
    \[\omega_{p'}([\RP^1]) = \omega_p([\RP^1]) \cdot (d+1\;d+2).\]
 \end{proposition}

\begin{remark}
    Sometimes it will be convenient to increase the degree of a given cover $p \colon F \rightarrow \mathbb{RP}^2$ without directly appealing to Proposition \ref{prop:stab}. More precisely, we will often work with covers defined as the gluing of a branched cover of $D^2$ and a connected regular $d$-fold cover of $Mb$. To get a $(d+2)$-fold cover with the same domain, we just stabilize the restriction to $D^2$ twice and change the monodromy of the core of $Mb$ to a cycle of length $d+2$. In any case, Theorem \ref{thm:gabai_kazez} implies that this alternative stabilization procedure gives a cover which is equivalent to the one defined in Proposition \ref{prop:stab}.
\end{remark}
 
\section{The mapping class group of a non-orientable surface}\label{sec:3}

Let $F$ be a surface. We let $\M(F)$ denote the \textsl{mapping class group} of $F$, i.e. the group of isotopy classes of homeomorphisms of $F$. We make the following further assumptions about elements of $\M(F)$, unless stated otherwise.

\begin{itemize}
    \item[-] When $\partial F\neq \emptyset$, we assume that elements of $\M(F)$ fix $\partial F$ pointwise.

    \item[-] When $F$ has a set of marked points $S$, we assume that elements of $\M(F)$ fix $S$ setwise. We may also denote this mapping class group as $\M(F,S)$, and its representatives as self-homeomorphisms of $(F,S)$.

    \item[-] When $F$ is orientable, we assume that elements of $\M(F)$ preserve a fixed orientation.
\end{itemize} 

We denote a Dehn twist along a simple 2-sided loop $\gamma$ by $t_\gamma$, and let $[t_\gamma]$ denote its class in the mapping class group $\M(F)$. Recall that if $F$ is non-orientable and $\gamma\subset F$ is one-sided, the Dehn twist along $\gamma$ is not defined. We let $\T(F)$ denote the \textsl{twist subgroup} of $\M(F)$, i.e. the subgroup of $\M(F)$ generated by Dehn twists.

It is a standard fact that, when $F$ is orientable and has no marked points, its mapping class group $\M(F)$ is generated by Dehn twists, i.e. $\M(F) = \T(F)$. On the other hand, when $F$ is non-orientable of genus $g(F)\geq 2$, $\T(F)$ is a proper subgroup of $\M(F)$ of index two, see \cite{Lickorish_1963, Lickorish_1965}. This fact was first shown in \cite{Lickorish_1963}, where Lickorish found an alternative generating set for the mapping class group $\M(K)$ of the Klein bottle $K$. More precisely, Lickorish showed that $\M(K)$ is generated by a Dehn twist along $\gamma$, the unique isotopy class of $2$-sided non-separating curves on $K$, along with a \textsl{$y$-homeomorphism}, which is defined on an embedded punctured Klein bottle. Up to the action of $\T(F)$, this diffeomorphism can be described as follows. We regard the Klein bottle $K$ as the mapping torus of a reflection on $S^1$.
The map
    \begin{align*}
    	I\times S^1 &\longrightarrow I\times S^1\\
    	(s,t) &\longmapsto (1-s,t)
    \end{align*}
    descends to an involution of the Klein bottle \[y_K \colon K\longrightarrow K.\]

    Let $D\subset K$ be an embedded 2-disk. Assume that $y_K$ has been isotoped to fix $D$ pointwise. Set $K^\circ = K \setminus \Int D$, and let \[y \colon K^\circ \to K^\circ\] be the restriction of $y_K$ to $K^\circ$. This defines a self-diffeomorphism of the punctured Klein bottle $K^\circ$, which we call a \textsl{$y$-homeomorphism}. 
    
    More generally, let $F$ be a surface containing an embedded copy of $K^\circ$. A self-ho\-me\-o\-mor\-phism of $F$ is a \textsl{$y$-homeomorphism} if it is isotopic to a homeomorphism that coincides with $y$ on the embedded $K^\circ$ and with the identity map on its complement.

The $y$-homeomorphism depends on a specified embedding of $K^\circ$ in $F$. There are often many different embeddings, and these give rise to different $y$-homeomorphisms. When $F \cong \RP^2$ there cannot exist an embedding $K^\circ \subset F$, and so there are no $y$-homeomorphisms on $F$. Indeed, in this case we have that $\M(F)$ is trivial. If $F$ is a closed connected non-orientable surface and $g(F)>1$, $\M(F)$ is generated by Dehn twists and a single $y$-homeomorphism, and $\T(F)$ forms an index two subgroup in $\M(F)$ \cite{Lickorish_1963}. In \cite{Chillingworth_1969}, Chillingworth found a finite set of generators for $\M(F)$, which we will largely use in Section \ref{sec:lift_chill}. See also \cite{BirmanChillingworth}, where this generating set is recovered from the mapping class group of the orientation double cover, and \cite{Korkmaz2002, Szepietowski2006} for later refinements.

\begin{theorem}[Chillingworth \cite{Chillingworth_1969}]\label{thm:chillingworth_69}
    Let $F$ be a closed non-orientable surface as described in Figure \ref{fig:chill_generators}, and set $N=\lfloor (g(F)-1)/2\rfloor$, so that $g(F)=2N+1$ if $g(F)$ is odd and $g(F)=2N+2$ if $g(F)$ is even. Then $\M(F)$ is generated by a single $y$-homeomorphism, plus Dehn twists along the following curves: \begin{enumerate}[(i)] 
        \item if $g(F)$ is even, as in Figure \ref{fig:chill_generators}(b), the curves $\alpha_i$, $\beta_i$ and $\gamma_i$, for $i=1, \dots, N$. 
        
        \item if $g(F)$ is odd, as in Figure \ref{fig:chill_generators}(a), the curves $\alpha_i$, $\beta_i$ and $\gamma_i$, for $i=1, \dots, N-1$.
    \end{enumerate}
\end{theorem}

 For our purposes, we do not need to specify the support of Chillingworth's generating $y$-homeomorphism, though that may be found in \cite{Chillingworth_1969}. It is worth noting that $\T(F)$ is \textsl{not} generated solely by Dehn twists along the curves listed in Theorem \ref{thm:chillingworth_69}. This will be further explored in Section \ref{sec:lift_chill}.

\begin{figure}
    \centering
     \begin{overpic}[width=0.75\linewidth]{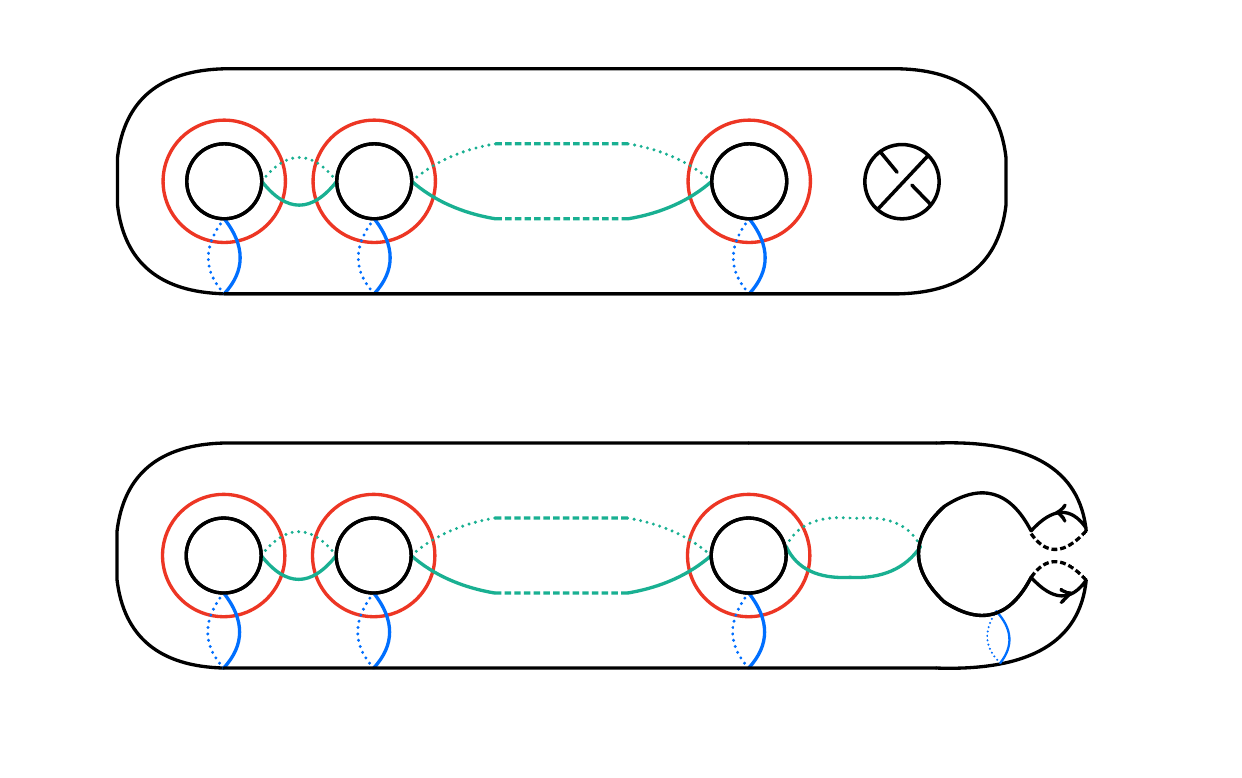}
        \put(12,53){\textcolor{red}{\small{$\alpha_1$}}}
        \put(24,53){\textcolor{red}{\small{$\alpha_2$}}}
        \put(54,53){\textcolor{red}{\small{$\alpha_N$}}}
        \put(13,41){\textcolor{blue}{\small{$\beta_1$}}}
        \put(25,41){\textcolor{blue}{\small{$\beta_2$}}}
        \put(55,41){\textcolor{blue}{\small{$\beta_N$}}}
        \put(22.5,44){\textcolor{PineGreen}{\small{$\gamma_1$}}}
        \put(35,44){\textcolor{PineGreen}{\small{$\gamma_2$}}}
        \put(48,43){\textcolor{PineGreen}{\small{$\gamma_{N-1}$}}}
        \put(12,23){\textcolor{red}{\small{$\alpha_1$}}}
        \put(24,23){\textcolor{red}{\small{$\alpha_2$}}}
        \put(54,23){\textcolor{red}{\small{$\alpha_N$}}}
        \put(13,11){\textcolor{blue}{\small{$\beta_1$}}}
        \put(25,11){\textcolor{blue}{\small{$\beta_2$}}}
        \put(55,11){\textcolor{blue}{\small{$\beta_N$}}}
        \put(22.5,14){\textcolor{PineGreen}{\small{$\gamma_1$}}}
        \put(35,14){\textcolor{PineGreen}{\small{$\gamma_2$}}}
        \put(48,13){\textcolor{PineGreen}{\small{$\gamma_{N-1}$}}}
        \put(67,14){\textcolor{PineGreen}{\small{$\gamma_N$}}}
        \put(72,11){\textcolor{blue}{\small{$\beta_{N+1}$}}}
    \end{overpic}
\caption{Chillingworth's generators for the mapping class group of a closed, non-orientable surface.}\label{fig:chill_generators}

\end{figure}

\section{Projection diagrams for the generators of the twist subgroup}\label{sec:4}

\begin{definition}
    Let $p\colon F\rightarrow \Sigma$ be a simple $d$-fold branched cover between surfaces. A \textsl{projection diagram} for $p$ is a picture of $\Sigma$, decorated with the branch set $B_p$ and a collection of embedded arcs $\{\alpha_i\}$, such that\begin{itemize}
        \item[-] $\Sigma \setminus (B_p \cup \bigcup_i \alpha_i)$ is connected and simply connected; the restriction of $p$ over it is hence a trivial $d$-fold cover, so that its $d$ sheets may be numbered;

        \item[-] each arc $\alpha_i$ is labeled with an element $\sigma_i\in S_d$, namely the permutation by which the sheet numberings on the two sides of $\alpha_i$ differ.
    \end{itemize} The collection $\{\alpha_i\}$ is a \textsl{cut system} for $p$, and will sometimes be denoted $C_p$. Equivalently, $F$ is recovered from $d$ copies of $\Sigma \setminus (B_p\cup \bigcup_i\alpha_i)$ by regluing the two sides of the $i$-th cut $\alpha_i$ according to $\sigma_i$. The labels $\sigma_i$ are determined by $\omega_p$ once the numbering of the sheets is fixed.
\end{definition}

Projection diagrams are frequently used for branched covers over $D^2$ or $S^2$. In these cases, the cut system is a collection of arcs that either connect two branch points with the same monodromy or connect a branch point to $\partial \Sigma$, and each arc is labeled with a transposition. We will be working with projection diagrams over $\RP^2$, and we will depict them by decomposing $\RP^2$ as $D^2\cup Mb$, with all the branch points contained in $D^2$. The cut system will associate transpositions to all the arcs on $D^2$, and a more general element of $S_d$ to a proper arc in $Mb$ dual to the core of the band; the label of that last arc is precisely $\omega_p([\RP^1])$.

Let $\gamma\subset F$ be a simple loop and suppose that its image $p(\gamma)\subset \Sigma$ is self-transversely immersed. Then $\gamma$ can be fully recovered from its projection $p(\gamma)$ by assigning to an arc in $p(\gamma)\setminus C_p$ a label corresponding to the number of the sheet where its pre-image lies. Though not entirely necessary, we sometimes label multiple arcs in $p(\gamma) \setminus C_p$ with their corresponding sheets for convenience.

\begin{remark}\label{rmk:stab_curves} 
     Given $p\colon F\to \Sigma$ and $\gamma\subset F$ as above, we will sometimes stabilize $p$ to perform certain isotopies to $\gamma$ through its image $p(\gamma)$ on the projection diagram. When performing these stabilizations, we will always assume that the new branch points, along with any cuts associated to them, are away from $p(\gamma)$. In particular, whenever we perform stabilizations with the intention of isotoping $\gamma$, we will do it in such a way that the isotopy class of $\gamma \subset F$ is unchanged.
 \end{remark}

\begin{example}\label{eg:diagram1}
    A classical starting point for simple branched covers between surfaces is the double cover associated to the hyperelliptic involution on $\Sigma_g$. Let \[\iota\colon \Sigma_g\longrightarrow \Sigma_g\] be the involution induced by a $\pi$ rotation of $\Sigma_g$ about the axis shown on the left in Figure \ref{fig:diagram1}. This gives rise to a simple branched cover \[p\colon \Sigma_g\longrightarrow S^2\] with $2g+2$ branch points. Since the branched cover is $2$-fold, all branch points have the same monodromy $(1\, 2)\in S_2$. A picture of the preferred projection diagram of $p$ is shown on the right in Figure \ref{fig:diagram1}, along with three loops. The purple loop on the projection diagram has a specified labeling, and so has one pre-image on $\Sigma_g$, which is shown on the left (also in purple). The green loop is unlabeled, and so its pre-image has two components, also shown on the left. For the blue curve, there does not exist a labeling compatible with $B_p$, so this curve does not correspond to a simple loop in $\Sigma_g$.

    \begin{figure}
        \centering
     \begin{overpic}[width=0.75\linewidth]{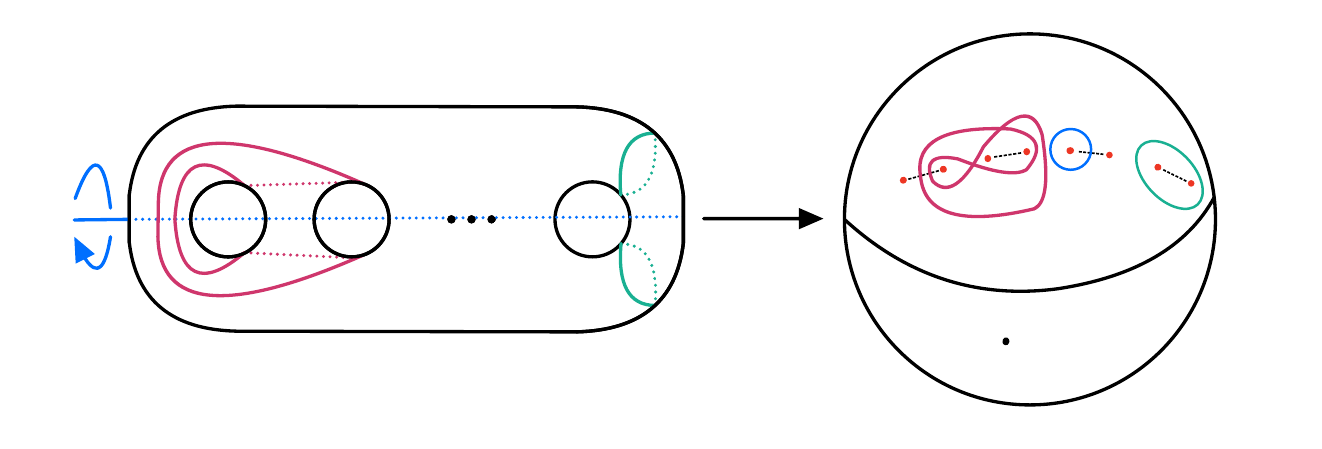}
        \put(7,11){\textcolor{blue}{$\iota$}}
        \put(68,23){\textcolor{purple}{\small{$1$}}}
    \end{overpic}
        \caption{Left: $\Sigma_g$ and the axis of rotation for $\iota$. Right: projection diagram for $p$, along with three curves. 
        All branch points shown have monodromy $(1\,2)$. The cut system is comprised of the dashed gray arcs.}
        \label{fig:diagram1}
    \end{figure}
\end{example}

From this point onward, we will continue constructing examples of projection diagrams for various branched covers. We will build off this first example, and gradually work towards the projection diagrams which will be our main tools in Section \ref{sec:lift_chill}.

\begin{example}\label{eg:diagram2}
    Take the cover $p\colon \Sigma_g\rightarrow S^2$ from Example \ref{eg:diagram1}. We modify $p$ in two different ways to obtain branched covers on $\Sigma_g^2=\Sigma_g \setminus \sqcup_2 \Int D^2$ and $\Sigma_g^1=\Sigma_g \setminus \Int D^2$. \begin{enumerate}
        \item First restrict $p$ to $\mathcal D_p$, the complement of a disk in $S^2$ which is disjoint from $B_p$. In this way, we obtain a $2$-fold cover \[p_1\colon \Sigma_g^2\rightarrow \mathcal D_p\] with $2g+2$ branch points. There exists a cut system for this branched cover such that all cuts are in the disk interior, and the monodromy associated to each cut is $(1\,2)$. Add two stabilizations to $p_1$, one on each boundary component of $\Sigma^2_g$, to obtain a $4$-fold branched cover with $2g+4$ branch points, which we continue to denote by $p_1$. The two new branch points have cuts connecting them to $\partial \mathcal D_p$ and have monodromy $(1 \, 3)$ and $(2\,4)$, respectively. After these stabilizations, $\omega_{p_1}(\partial \mathcal D_p)=(1\,3)(2\,4)$. A projection diagram for $p_1$ is shown on the left in Figure \ref{fig:even_p}.

        \item Next, restrict $p$ to $\mathcal D_q$, the complement of a disk in $S^2$ which contains one branch point in $B_p$. In this way, we obtain a $2$-fold cover \[q_1\colon \Sigma_g^1\rightarrow \mathcal D_q\] with $2g+1$ branch points. It admits a cut system with $g$ pairs of cuts in $\Int \mathcal D_q$ and one cut from the last remaining branch point to $\partial \mathcal D_q$. All cuts have $(1\,2)$ as their monodromy. We stabilize $q_1$ one time by adding a branch point with monodromy $(1\,3)$. With this, the monodromy for $\partial \mathcal D_q$ is $(1\,2\,3)$. A projection diagram for $q_1$ is shown on the left in Figure \ref{fig:odd_p}. 
    \end{enumerate}
\end{example}

\begin{figure}
    \centering
     \begin{overpic}[width=0.75\linewidth]{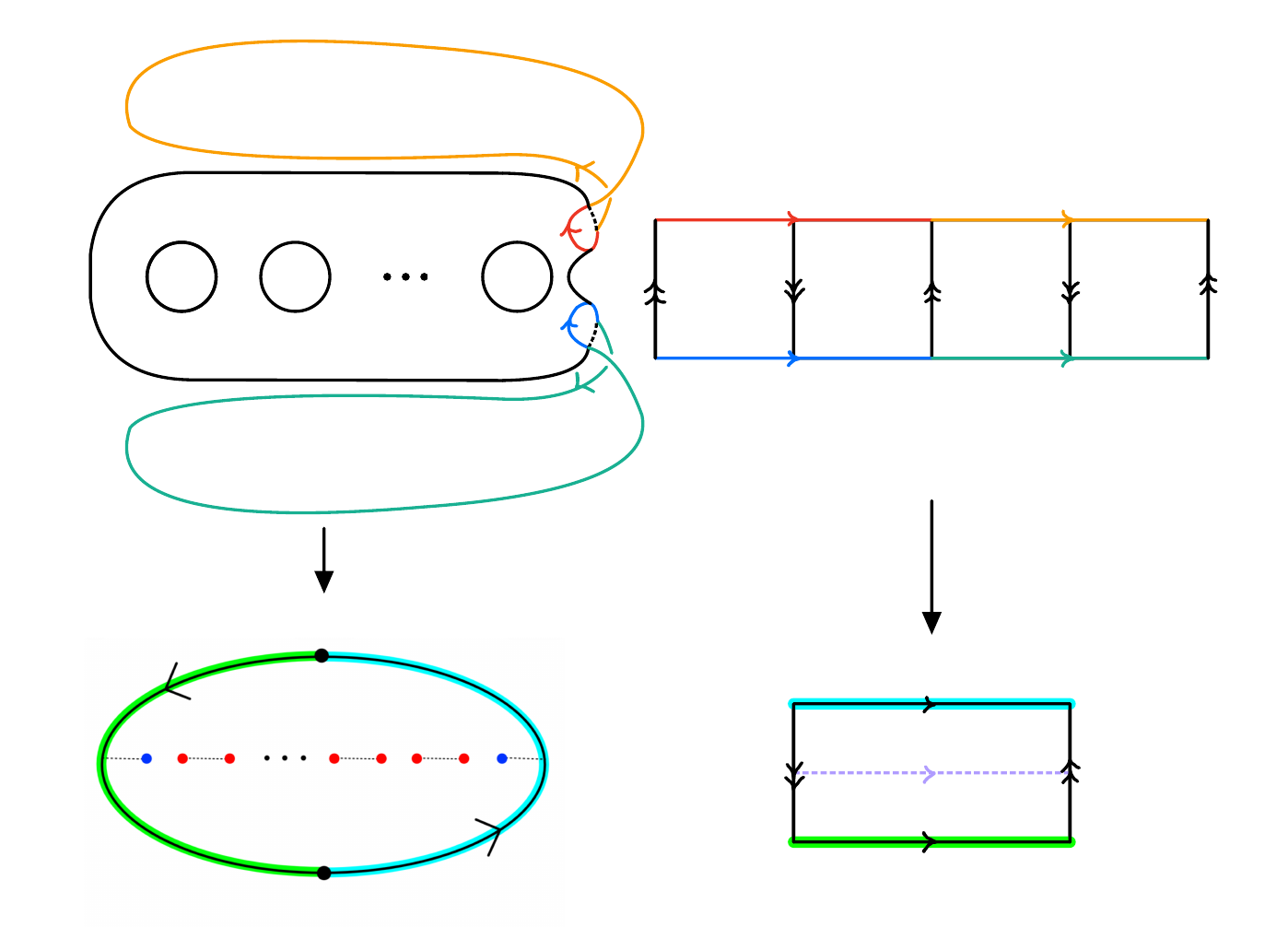}
        \put(65,8){\textcolor{violet}{$\mathbb{RP}^1$}}
        \put(0.5,10){\small{$(1\, 3)$}}
        \put(43.5,10){\small{$(2\, 4)$}}
    \end{overpic}
    \caption{The covers $p_1$ (left) and $p_2$ (right) which together form $p$. In the bottom left, the red branch points have monodromy $(1\,2)$ and the blue branch points have monodromy as given in the labels. In the upper half, the boundaries of the four sheets of each $p_i$ are colored. When $\Sigma_g^2$ is glued to $A$, each sheet of $p_1$ is glued to the sheet of $p_2$ with the same color. In the bottom half, the green and blue highlighting shows how $\partial \mathcal D_p$ is glued to $\partial Mb$.} \label{fig:even_p}

\end{figure}

    \begin{figure}
    \centering
     \begin{overpic}[width=0.75\linewidth]{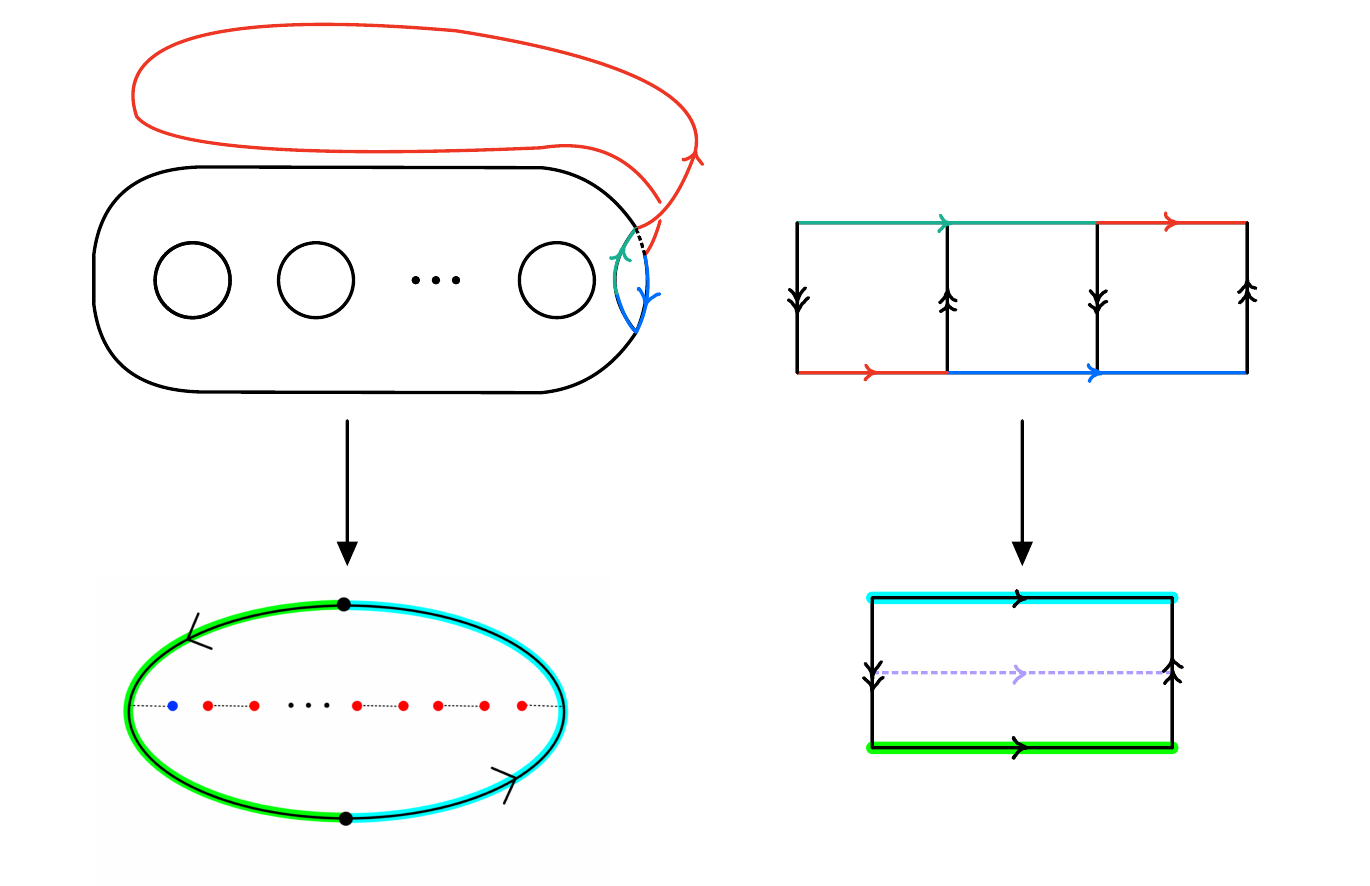}
        \put(69,11.5){\textcolor{violet}{$\mathbb{RP}^1$}}
        \put(3,10.5){\small{$(1\, 3)$}}
    \end{overpic}
    \caption{The covers $q_1$ (left) and $q_2$ (right) which together form $q$. In the bottom left, the red branch points have monodromy $(1\,2)$ and the blue branch points have monodromy as given in the labels. In the upper half, the boundaries of the three sheets of each $q_i$ are colored. When $\Sigma_g^1$ is glued to $Mb$, each sheet of $q_1$ is glued to the sheet of $q_2$ with the same color. In the bottom half, the green and blue highlighting shows how the boundary of $\mathcal D_q$ is glued to $Mb$.} \label{fig:odd_p}

\end{figure}

\begin{example}\label{eg:diagram3} Since this paper is chiefly concerned with branched covers over $\RP^2$, we will next build off the constructions in Example \ref{eg:diagram2} to obtain branched covers of closed non-orientable surfaces.
    \begin{enumerate}
        \item Let $p_1\colon \Sigma_g^2\rightarrow \mathcal D_p$ be the cover from Example \ref{eg:diagram2}(1) and let $p_2\colon A\rightarrow Mb$ be the standard, un-branched $4$-fold cover of an annulus over a Möbius band, i.e.\ the cover with $\omega_{p_2}([\RP^1])=(1\,2\,3\,4)$. 
Assume the sheets of $p_2$ have been labeled so that $\omega_{p_2}(\partial Mb) = \omega_{p_1}(\partial D_p)= (1\,3) (2\,4)$. We may glue $\partial \mathcal D_p$ to $\partial Mb$ and glue their pre-images $\partial \Sigma_g^2$ and $\partial A$ together to get a $4$-fold branched cover of $\RP^2$. The gluing of $\partial \Sigma_g^2$ and $\partial A$ will necessarily preserve orientation along one pair of $S^1$ components and reverse orientation along the other, and so the total space of this branched cover is the closed non-orientable surface $N_{2g+2}$. Let \[
            p\colon N_{2g+2}\longrightarrow \RP^2
        \] denote the branched cover obtained in this way.

        Now that the base is $\RP^2$, drawing a projection diagram is more complicated. We've constructed $p$ with the explicit purpose of mitigating this issue. Recall that $p$ is obtained by quotienting \[
            p_1\sqcup p_2 \colon \Sigma_g^2\sqcup A \longrightarrow \mathcal D_p \sqcup Mb
        \] along its boundaries, and that all of $B_p$ is contained in $\mathcal D_p$. We will adopt the convention that the projection diagram for $p$ is the same as the one for $p_1\sqcup p_2$, plus additional data about how $\partial \mathcal D_p$ and $\partial Mb$ are glued. The bottom of Figure \ref{fig:even_p} shows this projection diagram. The boundaries of $\mathcal D_p$ and $Mb$ are shaded and oriented to show how they are identified. On the top of Figure \ref{fig:even_p}, we see how the gluing of $\partial Mb$ and $\partial \mathcal D_p$ lifts to gluing of $\partial \Sigma_g^2$ and $\partial A$. If an arc $\delta\subset \RP^2$ intersects $\partial \mathcal D_p$, we will assume that it travels ``straight across'' the Möbius band and so re-enters $\mathcal D_p$ antipodally. From the upper-right portion of Figure \ref{fig:even_p}, we see that sheets $1$ and $2$, whose boundaries are colored red and blue, respectively, are ``straight across'' from each other on $A$, as are sheets $3$ and $4$, whose boundaries are orange and green. For this reason, if $\delta$ has labeling $1$ (resp. $3$) when it intersects $\partial \mathcal D_p$, it will re-enter $\mathcal D_p$ antipodally with labeling $2$ (resp. $4$).

        \item Let $q_1\colon \Sigma_g^1 \rightarrow \mathcal D_q$ be the cover from Example \ref{eg:diagram2}(2) and let $q_2\colon Mb\rightarrow Mb$ be the standard un-branched $3$-fold cover of a Möbius band over itself. Similar to last time, we glue $\partial \mathcal D_q$ to $\partial Mb$ and glue their pre-images $\partial \Sigma_g^1$ and $\partial Mb$ together to get a $3$-fold branched cover \[
            q\colon N_{2g+1}\longrightarrow \RP^2.
        \]
        The bottom of Figure \ref{fig:odd_p} shows the projection diagram for $q$, which is once again a decorated copy of $\mathcal D_q\sqcup Mb$. From the upper-right portion of Figure \ref{fig:odd_p}, we see that half of sheet 1 is antipodal to half of sheet 2. Therefore, if an arc $\delta\subset \RP^2$ goes through the blue-highlighted part of $\partial \mathcal D_q$ with labeling 2, it will re-enter antipodally with labeling 1.
    \end{enumerate}
\end{example}

The branched covers $p$ and $q$ in Example \ref{eg:diagram3} will be our main tool in Section \ref{sec:lift_chill}. One fact that we can already deduce about these maps is the following.

\begin{proposition}\label{prop:chill_DT_images}
    Let $F$ be a closed, non-orientable surface and let $h\in \{p,q\}$ be the branched cover from Example \ref{eg:diagram3} that has $F$ as a total space. Let $\delta$ be one of the generating curves for $\M(F)$ listed in Theorem \ref{thm:chillingworth_69}. Then, modulo isotoping $\delta$, $h(\delta)$ is an embedded loop which bounds an embedded disk in $\mathcal D_{h}$ containing at least two branch points in $B_{h}$ with monodromy $(1\,2)$.
\end{proposition}

We omit the proof of Proposition \ref{prop:chill_DT_images}, which is an exercise in computing the images of the curves in Figure \ref{fig:chill_generators} under $p$ and $q$. 

\begin{figure}[ht]
        \centering
        \begin{tikzpicture}
        \node[anchor=south west, inner sep=0] (fig) at (0,0) {\includegraphics[width=5in]{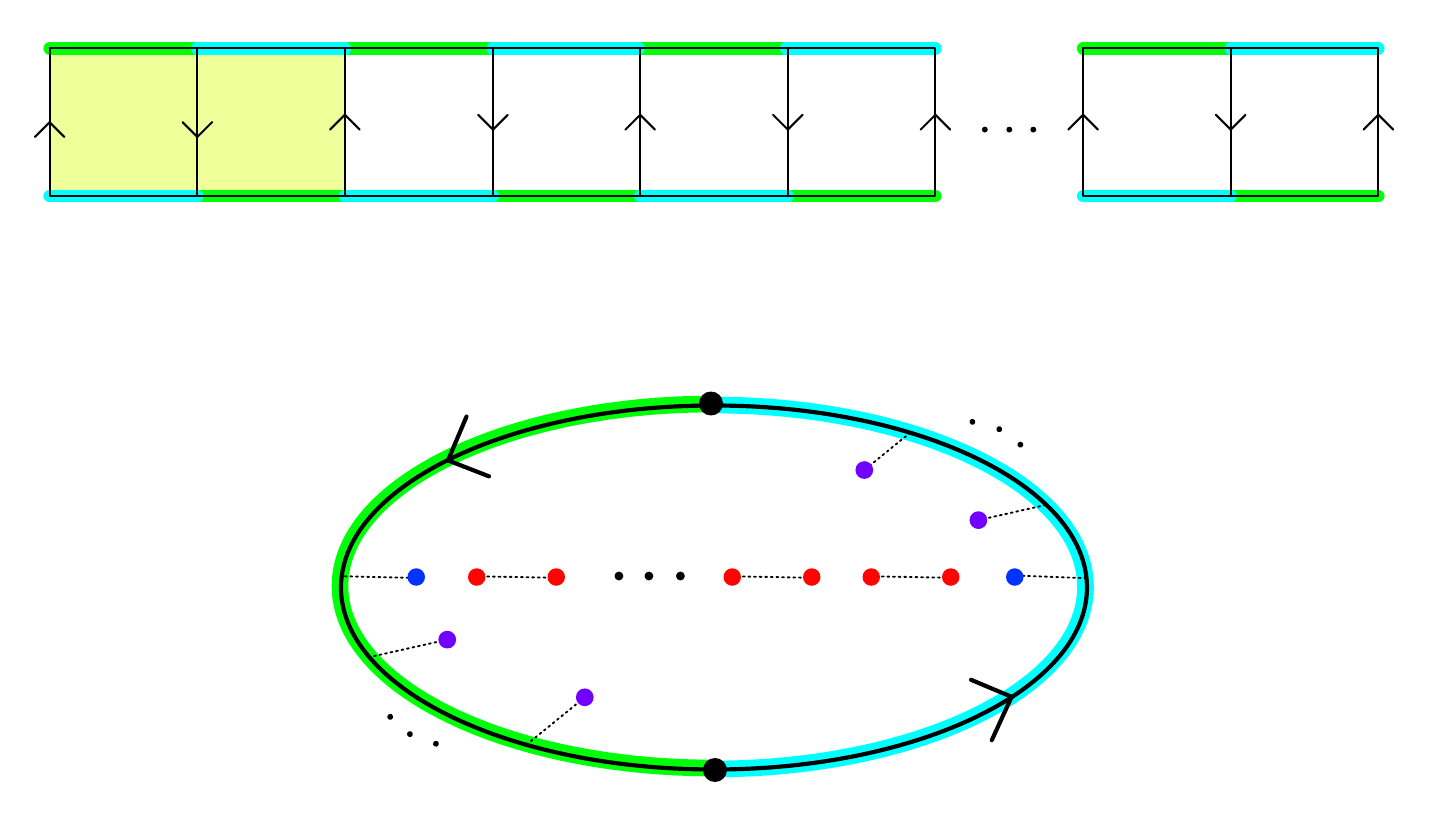}};
        \draw[decorate, decoration={brace, amplitude=10pt, mirror}] (.5,5.45) -- (3,5.45) node[midway, scale=.75, yshift=-20pt] {$1$};
        \draw[decorate, decoration={brace, amplitude=10pt, mirror}] (3.1,5.45) -- (5.6,5.45) node[midway, scale=.75, yshift=-20pt] {$3$};
        \draw[decorate, decoration={brace, amplitude=10pt, mirror}] (5.7,5.45) -- (8.3,5.45) node[midway, scale=.75, yshift=-20pt] {$5$};
        \draw[decorate, decoration={brace, amplitude=10pt, mirror}] (9.6,5.45) -- (12.3,5.45) node[midway, scale=.75, yshift=-20pt] {$2k-1$};

        \draw[decorate, decoration={brace, amplitude=10pt}] (.5,7.05) -- (3,7.05) node[midway, scale=.75, yshift=20pt] {$2$};
        \draw[decorate, decoration={brace, amplitude=10pt}] (3.1,7.05) -- (5.6,7.05) node[midway, scale=.75, yshift=20pt] {$4$};
        \draw[decorate, decoration={brace, amplitude=10pt}] (5.7,7.05) -- (8.3,7.05) node[midway, scale=.75, yshift=20pt] {$6$};
        \draw[decorate, decoration={brace, amplitude=10pt}] (9.6,7.05) -- (12.3,7.05) node[midway, scale=.75, yshift=20pt] {$2k$};

        \node[scale=.75] at (2.6,2.2) {$(1\,3)$};
        \node[scale=.75] at (10.1,2.2) {$(2\,4)$};
        \node[scale=.75] at (2.8,1.3) {$(1\,5)$};
        \node[scale=.75] at (4.5,.4) {$(1\, 2k-1)$};
        \node[scale=.75] at (9.7,3) {$(2\,6)$};
        \node[scale=.75] at (8.2,3.8) {$(2\, 2k)$};


        \end{tikzpicture}
        
        \caption{The yellow rectangle shows that the intersections of sheets $1$ and $2$ with $\partial(\Sigma_g^2)$ remain antipodal to each other after each pair of stabilizations, where sheets $1$ and $2$ are the only sheets that are not stabilizing disks.}
        \label{fig:even_stabilize}
    \end{figure}

\begin{remark} \label{rmk:stabilizations} We provide explicit stabilizations for the covers constructed in Example \ref{eg:diagram3}, along with their corresponding projection diagrams. Taking Theorem \ref{thm:gabai_kazez} into consideration, this collection of stabilized maps exhausts all equivalence classes of simple branched covers from closed, non-orientable surfaces to $\RP^2$.\begin{enumerate}
    \item Consider $p\colon N_{2g+2}\rightarrow \RP^2$ from Example \ref{eg:diagram3}(1). In the construction of $p$, we applied two stabilizations to $p_1$: one with monodromy $(1\,3)$ and the other with monodromy $(2\,4)$. We may, if we wish, stabilize $p_1$ by adding two more branch points with monodromy $(1\,5)$ and $(2\, 6)$. We then obtain a $6$-fold cover $N_{2g+2}\rightarrow \RP^2$ by gluing $p_1$ to the $6$-fold cover $A\rightarrow Mb$. Iterating on this process, we see that we can always apply two stabilizations to $p_1$, and extend it to a cover $N_{2g+2}\rightarrow \RP^2$. In this way, we obtain a $2k$-fold branched cover $N_{2g+2}\rightarrow \RP^2$ for all $k\geq 1$. Figure \ref{fig:even_stabilize} shows the base diagram $\mathcal D_p$ stabilized in this way. The total monodromy around $\partial \mathcal D_p$ is $(1\,3\,5\,\cdots 2k-1)(2\,4\,\cdots 2k)$. After extending $p_1$ to a cover $p\colon N_{2g+2}\rightarrow \RP^2$, this is the monodromy associated to $\ell^2$, where $\ell=[\RP^1]\in\pi_1(\RP^2\setminus B_p)$; equivalently $\omega_p(\ell)=(1\,2\,3\,\cdots\,2k)$.

    The top of Figure \ref{fig:even_stabilize} shows how the boundaries on each of the $2k$ sheets will be glued along $\partial A$. The stabilizations have been arranged so that with each pair of stabilizations, sheets $1$ and $2$ continue to be antipodal to each other. Therefore, whenever a curve on our labeled diagram with label $1$ (resp. $2$) hits the boundary of $\mathcal D_p$, it will re-enter $\mathcal D_p$ on the antipodal side with label $2$ (resp. 1). Given a simple loop $\gamma\subset N_{2g+2}$, we will always assume that these stabilizations and cuts are performed away from its image $p(\gamma)\subset \RP^2$ (cf. Remark \ref{rmk:stab_curves}).

    \item Consider $q\colon N_{2g+1}\rightarrow \RP^2$ from Example \ref{eg:diagram3}(2). Just like the previous case, we establish a stabilization scheme for $q_1\colon \Sigma_g^1 \rightarrow \mathcal D_q$ so that it continues to extend over a cover $Mb\rightarrow Mb$. The following system does this: add two branch points with cuts antipodal to each other on $\mathcal D_q$, one with monodromy $(1\,4)$ and the other with monodromy $(1\,5)$. Working your way counterclockwise, add additional pairs of stabilizing branch points and cuts in this fashion. See the bottom of Figure \ref{fig:odd_stabilize} for a precise image of the stabilized labeled diagram. The total monodromy of $\partial \mathcal D_q$ after a finite number of stabilizations in this fashion will be a cycle of maximal length.

Now, this cover may be glued to the $(2k+1)$-fold cover $Mb\rightarrow Mb$, giving in total a $(2k+1)$-fold cover $N_{2g+1}\rightarrow \RP^2$. The top of Figure \ref{fig:odd_stabilize} shows how the boundaries on each of the $2k+1$ sheets will be glued along $\partial Mb$. The stabilizations have been arranged in such a way that the blue-highlighted edge of sheet 2 remains antipodal to the green-highlighted edge of sheet 1. Therefore, if any oriented arc on our labeled diagram has label 2 and intersects $\partial \mathcal D_q$ along the blue highlighted edge, it will re-enter $\mathcal D_q$ antipodally with a $1$ label. Conversely, if the arc has a $1$ label when it intersects the green edge of $\partial \mathcal D_q$, it will re-enter antipodally with a $2$ label. Given a simple loop $\gamma\subset F$, we will always assume that these stabilizations are performed away from its image $q(\gamma)\subset \RP^2$ (cf. Remark \ref{rmk:stab_curves}).    
\end{enumerate}
\end{remark}

\begin{figure}[ht]
        \centering

        \begin{tikzpicture}
            \node[anchor=south west, inner sep=0] (fig) at (0,0) {\includegraphics[width=6in]{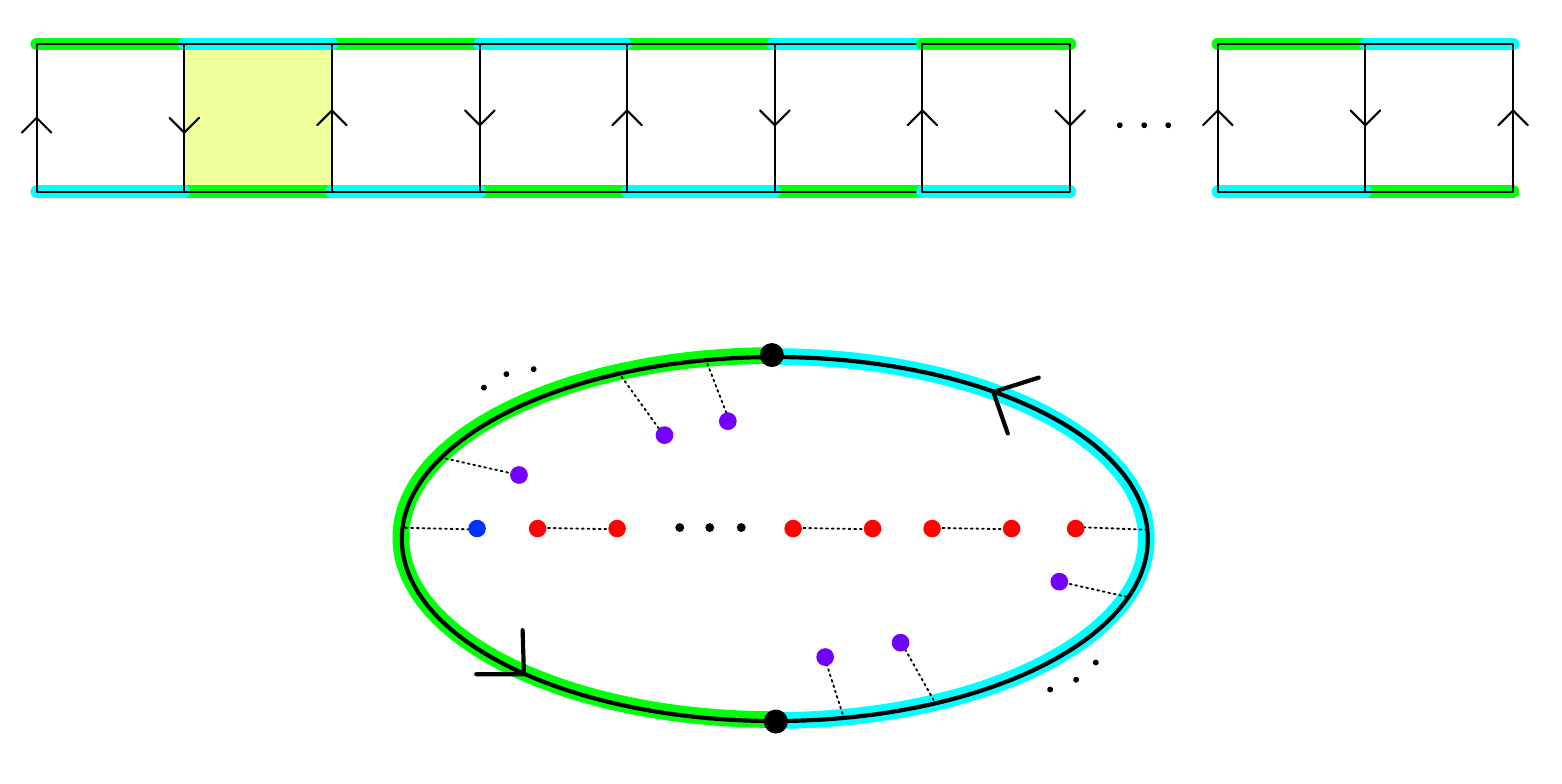}};
    
            \draw[decorate, decoration={brace, amplitude=10pt, aspect=0,mirror}] (.3,5.65) node[below, scale=.75, yshift=-13pt] {$3$}  -- (1.75,5.65) ;
            \draw[decorate, decoration={brace, amplitude=10pt, mirror}] (1.85,5.65) -- (4.6,5.65) node[midway, scale=.75, yshift=-20pt] {$1$};
            \draw[decorate, decoration={brace, amplitude=10pt, mirror}] (4.7,5.65) -- (7.5,5.65) node[midway, scale=.75, yshift=-20pt] {$4$};
            \draw[decorate, decoration={brace, amplitude=10pt, mirror}] (7.6,5.65) -- (10.4,5.65) node[midway, scale=.75, yshift=-20pt] {$6$};
            \draw[decorate, decoration={brace, amplitude=10pt, mirror}] (11.9,5.65) -- (14.7,5.65) node[midway, scale=.75, yshift=-20pt] {$2k$};
    
            \draw[decorate, decoration={brace, amplitude=10pt}] (.4,7.35) -- (3.2,7.35) node[midway, scale=.75, yshift=20pt] {$2$};
            \draw[decorate, decoration={brace, amplitude=10pt}] (3.3,7.35) -- (6.1,7.35) node[midway, scale=.75, yshift=20pt] {$5$};
            \draw[decorate, decoration={brace, amplitude=10pt}] (6.2,7.35) -- (9,7.35) node[midway, scale=.75, yshift=20pt] {$7$};
            \draw[decorate, decoration={brace, amplitude=10pt, aspect=1}] (9.1,7.35) -- (10.4,7.35) node[above, scale=.75, yshift=13pt] {$9$};
            \draw[decorate, decoration={brace, amplitude=10pt, aspect = 0}] (11.9,7.35) node[above, scale=.75, yshift=13pt] {$2k+1$} -- (13.3,7.35);
            \draw[decorate, decoration={brace, amplitude=10pt, aspect=1}] (13.4,7.35) -- (14.8,7.35) node[above, scale=.75, yshift=13pt] {$3$};
    
            \node[scale=.75] at (3.5,2.5) {$(1\,3)$};
            \node[scale=.75] at (8.3,.3) {$(1\,4)$};
            \node[scale=.75] at (9.3,.45) {$(1\,6)$};
            \node[scale=.75] at (11.5,1.7) {$(1\,2k)$};
            \node[scale=.75] at (6.8,4.45) {$(1\, 5)$};
            \node[scale=.75] at (6,4.3) {$(1\, 7)$};
            \node[scale=.75] at (3.4,3.2) {$(1\,\,\, 2k+1)$};

            \end{tikzpicture}
        \caption{The yellow-shaded square shows that half of the boundary for sheet 1 is antipodal to half the boundary of sheet 2, where sheets $1$ and $2$ are the only sheets that are not stabilizing disks.}
        \label{fig:odd_stabilize}
    \end{figure}
 
\section{Lifting homomorphism}\label{sec:5}
 
\begin{definition}[Braid group] Let $F$ be a surface and $p_1, \dots, p_n \subset \text{Int } F$ be pairwise distinct points. A \textsl{braid with $n$ strands over $F$} is an $n$-tuple $\Gamma=(\gamma_1, \dots, \gamma_n)$ satisfying the following conditions:
\begin{itemize}
    \item[-] $\gamma_i \colon [0,1]\rightarrow F$ is a continuous map such that $\gamma_i(0)=p_i$ for every $i=1, \dots, n$;
    \item[-] $\gamma_i(1)\in \{p_1, \dots, p_n\}$ for every $i=1, \dots, n$;
\item[-] $\{\gamma_1(t), \dots, \gamma_n(t)\}$ are $n$ distinct points in $F$ for every $t \in [0,1]$.
\end{itemize}

The \textsl{braid group of $n$ strands on $F$} is obtained from the set of homotopy classes rel endpoints of braids of $n$ strands over $F$, by endowing it with the group structure induced by the concatenation of $n$-tuples of paths. Equivalently, \[\B_n(F)=\pi_1\big(\Gamma_n(F),\{p_1,\dots,p_n\}\big),\] where $\Gamma_n(F)$ is the space of $n$-element subsets of $\mathrm{Int}\,F$.
\end{definition}

Given a closed surface $F$, the Birman exact sequence \cite{Birman} reads
    \[\pi_1\big(\mathrm{Diff}(F),\id\big) \longrightarrow \B_n(F)\longrightarrow \M(F , \{p_1, \dots, p_n\}) \longrightarrow \M(F) \longrightarrow 1,\]
relating the group of braids over $F$ with the mapping class group of $F$ with $n$ marked points. When $\chi(F)<0$ the group $\mathrm{Diff}(F)$ has contractible components and the left-hand map is trivial, so that $\B_n(F)$ injects into $\M(F,\{p_1,\dots,p_n\})$. This is \emph{not} the case for $F=\RP^2$: since $\mathrm{Diff}(\RP^2)\simeq SO(3)$ we have $\pi_1(\mathrm{Diff}(\RP^2))\cong \Z_2$, and $\M(\RP^2)$ is trivial, so the sequence becomes
    \[1 \longrightarrow \Z_2 \longrightarrow \B_n(\RP^2)\longrightarrow \M(\RP^2, \{p_1, \dots, p_n\}) \longrightarrow 1 .\]
Here the $\Z_2$ is generated by the full twist $\Delta_n$, which is the unique element of order two in $\B_n(\RP^2)$ \cite{vanBuskirk, GG2004}; thus
    \[\M(\mathbb{RP}^2, \{p_1, \dots, p_n\}) \;\cong\; \B_n(\mathbb{RP}^2)/\langle \Delta_n\rangle .\]

\begin{definition}[Half twist]
    Let $F$ be a surface with a fixed $n$-tuple of pairwise distinct points $\mathcal{P}=\{p_1, \dots, p_n\}$ and let $\alpha \subset F$ be  a simple arc meeting $\mathcal{P}$ only at its end points. Without loss of generality, assume that $\alpha \cap \mathcal{P}=\{p_1, p_2\}$. A \textsl{half twist} on $F$ around $\alpha$ is a braid $h_{\alpha}=(\gamma_1, \dots, \gamma_n)$ such that
\begin{itemize}
        \item[-] $\gamma_i$ is constantly equal to $p_i$ for $i \geq 3$;
        \item[-] $\gamma_1$ and $\gamma_2$ are supported in a disk neighborhood $D$ of $\alpha$ containing no other point of $\mathcal P$, and $h_\alpha$ is obtained by rotating $D$ by an angle $\pi$ about the midpoint of $\alpha$; in particular $\gamma_1(1)=p_2$ and $\gamma_2(1)=p_1$.
    \end{itemize}
    
    \begin{figure}
    \centering
     \begin{overpic}[width=0.75\linewidth]{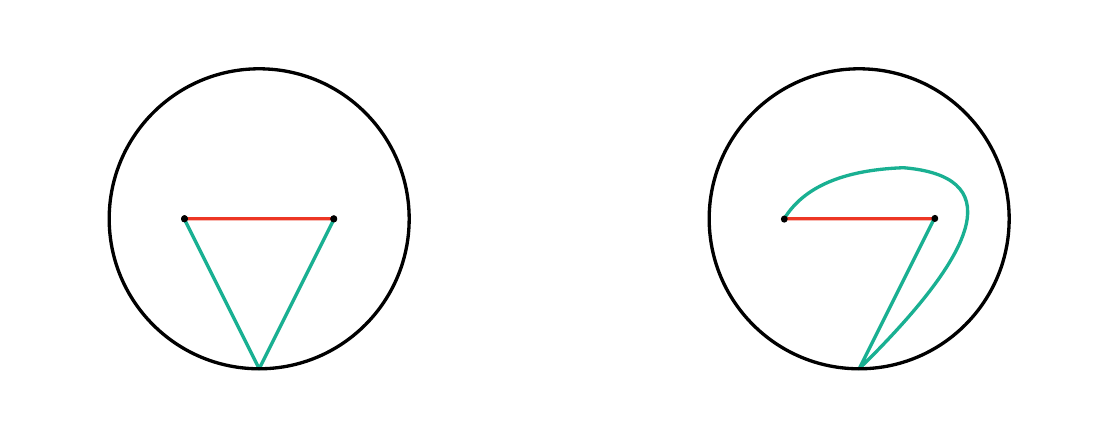}
        \put(22,17){\textcolor{red}{\small{$\alpha$}}}
        \put(73,16){\textcolor{red}{\small{$h_{\alpha}(\alpha)$}}}
        \put(16,12){\textcolor{PineGreen}{\small{$\sigma_1$}}}
        \put(27.5,12){\textcolor{PineGreen}{\small{$\sigma_2$}}}
        \put(70,12){\textcolor{PineGreen}{\small{$h_{\alpha}(\sigma_1)$}}}
        \put(70,26){\textcolor{PineGreen}{\small{$h_{\alpha}(\sigma_2)$}}}
    \end{overpic}
    \caption{The action of the half twist $h_{\alpha}$ on the simple arcs $\alpha$, $\sigma_1$ and $\sigma_2$.} \label{fig:halftwist}

\end{figure}
\end{definition}

 Figure \ref{fig:halftwist} shows a local model of the half twist $h_\alpha$ in a neighborhood of $D$. We are now ready to introduce the notion of lifting homomorphism. 
\begin{definition}\label{def:lifthom}
    Let $p\colon F_1 \rightarrow F_2$ be a simple $d$-fold branched cover between surfaces. Let $b$ be a self-diffeomorphism of $F_2$ fixing the branch set $B_p$ set-wise, where $n=\#B_p\geq 1$. Then $b$ is \textsl{liftable with respect to $p$} if and only if there exists a self-diffeomorphism $\widetilde b$ of $F_1$ making the following diagram commute
    \begin{center}
\begin{tikzcd}
F_1 \arrow[d, "p"'] \arrow[r, "\widetilde b"] & F_1 \arrow[d, "p"] \\
F_2 \arrow[r, "b"']    & F_2.   
\end{tikzcd}
    \end{center} 
    If such $\widetilde b$ exists, it is unique up to composition with deck transformations of $p$. In particular, it is unique whenever the deck group of $p$ is trivial, which is the case for a simple branched cover of degree $d\geq 3$ with connected total space and at least one branch point. For $d=2$ the deck group is $\Z_2$, and in this case the lifting should be understood as being defined only up to the hyperelliptic involution. We say that a braid in $\B_n(F_2)$ is liftable with respect to $p$ if it can be represented by a self-diffeomorphism of $(F_2,B_p)$ which is liftable with respect to $p$.
    
    Denote by $\L_p$ the subgroup of liftable braids with respect to $p$. We can hence define the lifting homomorphism
    \[\lambda_p \colon \L_p \longrightarrow \M(F_1)\] by sending each braid to the isotopy class of its unique lift.
\end{definition}

We now state a general criterion for detecting the liftability of a self-diffeomorphism of $F_2$. The following facts are known to experts in this subject. We include a proof sketch for convenience.

\begin{proposition}[Lifting criterion]\label{prop:liftcrit}
    Let $p\colon F_1 \rightarrow F_2$ be a simple branched cover between surfaces whose branch set consists of $n$ points. A base point preserving self-diffeomorphism $b$ of $(F_2, B_p)$ is liftable with respect to $p$ if and only if 
    \[\omega_p \circ b_* = \sigma \, \omega_p \, \sigma^{-1},\] where $\sigma \in S_d$ and \[b_*\colon \pi_1(F_2 \setminus B_p)\longrightarrow \pi_1(F_2 \setminus B_p)\] is the map induced by $b$ on $\pi_1(F_2 \setminus B_p)$. 
    
    In particular, a half twist $h_\alpha$ around a simple arc $\alpha \subset F_2$ joining two distinct points of $B_p$ is liftable precisely when $p\inv(\alpha)$ contains a simple closed loop $\gamma$. In this case, the lift of $h_{\alpha}$ is the Dehn twist $t_\gamma$, up to the choice of the direction of the twist, which is not canonical when $F_1$ is non-orientable.
\end{proposition}

\begin{proof}[Idea of proof]
    To show the first part of the statement, we notice that the map $b \colon (F_2, B_p) \rightarrow (F_2, B_p)$ is determined by its restriction $b|:= b|_{F_2 \setminus B_p}$ and $b$ is liftable if and only if $b|$ is liftable with respect to the (un-branched ordinary) cover $p|:= p|_{F_1\setminus p\inv(B_p)}$. From the theory of covering spaces, we know that $b|$ is liftable if and only if $b|_*\circ p|_*$ is conjugate to $p|_*$, where $p|_*\colon \pi_1(F_1\setminus p\inv(B_p))\rightarrow \pi_1(F_2 \setminus B_p)$ is the induced map on $\pi_1$. The image of $p|_*$ is canonically isomorphic to the subgroup of $\pi_1(F_2 \setminus B_p)$ determined by the monodromy representation $\omega_p$. So the above condition is equivalent to asking that $\omega_p\circ b_*$ is conjugate to $\omega_p$.

    Now we discuss the second part of the statement, that a half twist $h_{\alpha}$ is liftable precisely when $p\inv(\alpha)$ contains a simple closed loop $\gamma$. The idea here is that if $\alpha\subset F_2$ is an arc between two points $a,b\in B_p$, the half twist $h_\alpha$ interchanges the meridians $\mu_a$ and $\mu_b$ of $a$ and $b$ respectively, up to conjugation. Since $p$ is simple, $\omega_p(\mu_a)$ and $\omega_p(\mu_b)$ are both transpositions. Since $\omega_p\circ {h_{\alpha}}_* = \omega_p$, $\omega_p(\mu_a)=\omega_p(\mu_b)=(i\,j)$ for some $i,j\in \{1,2,\cdots , d\}$. So basically, the sheets $i$ and $j$ of the cover $p|$ are pinched together at two different points, one in $p\inv(a)$ and the other in $p\inv(b)$. The curve $\gamma$ is formed by gluing together two arcs between these two pinch points: one along the $i$-sheet, the other along the $j$-sheet.
\end{proof}

\begin{remark}\label{rmk:ht_dt}
    One consequence of the above proof is the following: let $\gamma\subset F_1$ be a simple loop. Suppose that $p|_\gamma$ is injective and that $p(\gamma)$ bounds an embedded disk $D\subset F_2$ which contains exactly two branch points, whose monodromies are then necessarily the same transposition $(i\,j)$. Then $p^{-1}(D)$ is the union of $d-2$ disks and one annulus with core $\gamma$, and $t_\gamma=\lambda_p(h_\alpha)^{\pm 1}$, where $\alpha$ is an arc in $D$ between those branch points such that $p(\gamma)$ bounds a neighborhood of $\alpha$.
\end{remark}

\begin{remark}\label{rmk:fulltwist}
    The failure of injectivity for the map $\B_n(\RP^2) \rightarrow \M(\RP^2,\{p_1, \dots, p_n\})$ given by the Birman exact sequence is harmless for our purposes. Indeed, $\Delta_n$ acts trivially on $\pi_1(\RP^2\setminus B_p)$, so by Proposition \ref{prop:liftcrit} it is liftable for every $p$ and $\lambda_p(\Delta_n)=1$. Hence $\lambda_p$ factors through $\M(\RP^2,B_p)$, and the image of $\lambda_p$ --- which is all we are concerned with --- is unaffected. Nevertheless, statements such as ``$b$ is a braid'' should be read as ``$b$ is the mapping class of $\RP^2$ with marked points $B_p$ determined by a braid''.
\end{remark}

We relate the lifting homomorphisms of equivalent branched covers in the following lemma, the proof of which follows immediately from Theorem \ref{thm:gabai_kazez}.
\begin{lemma}\label{lem:lifting}
    Let $F$ be a closed connected surface and $p,q \colon F \rightarrow \mathbb{RP}^2$ be two simple branched covers of the same degree. Then $p$ is equivalent to $q$ via a commutative diagram
    \begin{center}
\begin{tikzcd}
F \arrow[d, "p"'] \arrow[r, "\Phi"] & F \arrow[d, "q"] \\
\mathbb{RP}^2 \arrow[r, "\phi"']    & \mathbb{RP}^2,   
\end{tikzcd}
    \end{center}
    where $\phi$ and $\Phi$ are diffeomorphisms and
    \[\lambda_q(\phi \circ b\circ \phi^{-1})=\Phi \circ \lambda_p(b)\circ \Phi^{-1}\]
    for any liftable braid $b \in \L_p$. In particular, $\lambda_p$ surjects onto the Dehn twist subgroup $\T(F)$ if and only if $\lambda_q$ does. The same statement holds with $\T(F)$ replaced by $\M(F)$, or by any subset of $\M(F)$ invariant under conjugation.
\end{lemma}

Lastly, we introduce the notion of \textsl{liftable isotopies} of branched covers.

\begin{definition}
    Let $p\colon F_1\to F_2$ be a simple branched cover. A \textsl{liftable isotopy} is an ambient isotopy $\Phi\colon F_2 \times I \to F_2$ which fixes the branch set $B_p$ pointwise and lifts to an ambient isotopy $\widetilde\Phi\colon F_1 \times I\to F_1$ with respect to $p$, namely $p\circ \widetilde \Phi = \Phi \circ (p \times \id_I)$.
\end{definition}

In Section \ref{sec:lift_chill}, we study the image of a simple loop $\gamma\subset F_1$ under $p$. We will often need to isotope $p(\gamma)$ in order to argue that $t_\gamma$ lies in the image of $\lambda_p$. The isotopies we apply to $p(\gamma)$ correspond to actual isotopies of $\gamma$ precisely when they are liftable. More detailed information about liftable isotopies can be found in \cite{Zuddas2007RepresentingDT}, where the third author provides a complete set of moves for liftable isotopies of branched covers over $D^2$.

\section{Surjectivity of the lifting homomorphism}\label{sec:lift_chill}

\subsection{The lifting homomorphism surjects onto the twist subgroup}
The aim of this section is to prove the following.
 \begin{theorem}\label{thm:chil}
    Let $F$ be a closed connected non-orientable surface. For any integer $d \geq 5$ such that $d \equiv \chi(F) \pmod{2}$, there exists a simple $d$-fold branched cover $p \colon F \rightarrow \RP^2$ such that the image of the lifting homomorphism
    \[\lambda_p \colon \L_p \longrightarrow  \M(F) \]
    contains the twist subgroup $\T(F)$. 
\end{theorem}

Let $F$ be a closed connected non-orientable surface of genus $g(F)=k$. In what follows, we identify a generating set for $\T(F)$ and show that each generator of $\T(F)$ is the lift of a half twist with respect to the branched covers $p$ and $q$ constructed in Example \ref{eg:diagram3}. The maps $p$ and $q$ themselves, along with their stabilizations (cf. Remark \ref{rmk:stabilizations}) will be used abundantly in what follows.

\subsection{Case 1: $g(F)=k$ is even}\label{sec:k_even}

Let $p\colon F\rightarrow \RP^2$ be the $4$-fold cover from Example \ref{eg:diagram3}(1). We show that, after a single pair of stabilizations, $\lambda_p$ surjects onto $\T(F)$. From \cite{Chillingworth_1969}, we know that $\T(F)$ is generated by Dehn twists along the following 2-sided simple loops:\begin{itemize}
    \item[-] the collection of loops given in Theorem \ref{thm:chillingworth_69} (i), shown in Figure \ref{fig:chill_generators}. We will re-label these curves and denote the entire collection by $\{\alpha_i\}$. Note that each $\alpha_i$ is contained in the interior of $\Sigma_g^2=F \setminus \sqcup_2 \Int D^2$;

    \item[-] the loops $\eta, \psi,$ and $\chi$, which we define below.
\end{itemize}

We may also regard $F$ as being obtained from $\Sigma_g^2=F \setminus \sqcup_2 \Int D^2$ by gluing the boundary components via a reflection of $S^1$. With this in mind, the curves $\eta$, $\psi$, and $\chi$ are shown in Figure \ref{fig:epc}. The image of each of these three curves under $p$ is shown in (1) of Figures \ref{fig:even_psi}, \ref{fig:even_chi}, and \ref{fig:even_eta}, respectively.

   \begin{figure}
    \centering
     \begin{overpic}[width=6in]{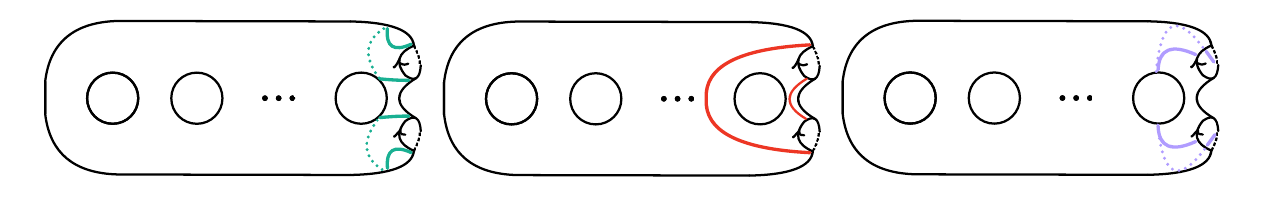}
         \put(27,12){\textcolor{PineGreen}{$\eta$}}
         \put(54,11){\textcolor{red}{$\psi$}}
         \put(88,12){\textcolor{violet}{$\chi$}}
        
    \end{overpic}
    \caption{The curves $\eta$, $\psi$, and $\chi$ in the even genus case.} \label{fig:epc}

\end{figure}

\begin{proposition}\label{prop:even_twist_lift}
    Up to stabilization, each generator of $\T(F)$ is the lift of a half twist with respect to $p$. In particular, it is enough to stabilize $p$ up to degree $6$.
\end{proposition}

To prove this, we first show a technical lemma.

\begin{lemma} \label{lem:bigon}
    Consider the branched cover of $D^2$ and the labeled projection diagram described on the left side of Figure \ref{fig:bigon_lem}. Then, after one stabilization, the diagram is isotopic through liftable isotopies to the right side of Figure \ref{fig:bigon_lem}.

    \begin{figure}
        \centering
      
        \begin{tikzpicture}
         \node[anchor=south west, inner sep=0] (fig) at (0,0) 
           {\includegraphics[width=6in]{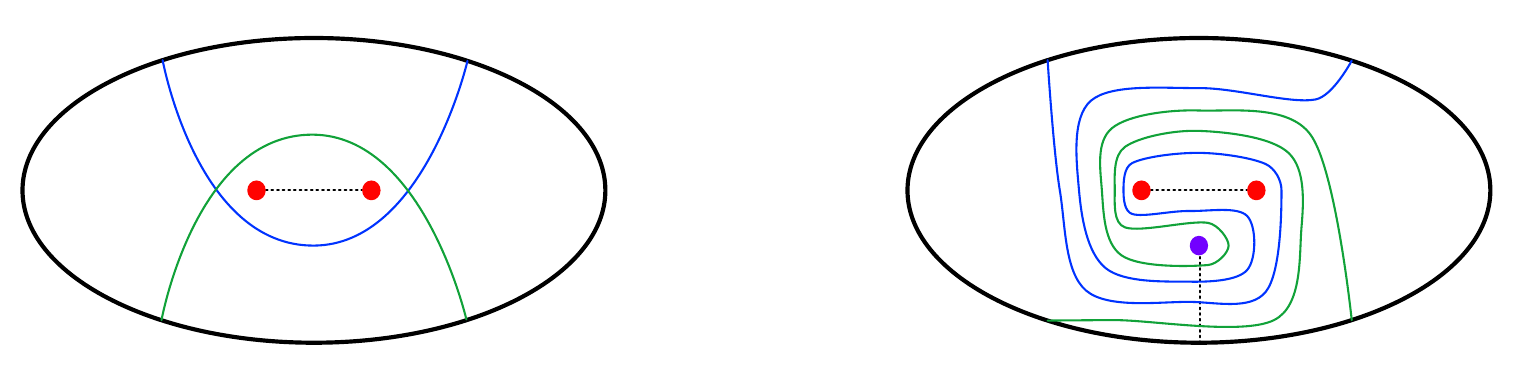}};
           
            \node at (7.5, 2) {\scalebox{2}{$\rightarrow$}};
            \node[ForestGreen] at (1.5,.25) {$b$};
            \node[blue] at (1.5,3.5) {$a$};

            \node[ForestGreen] at (10.5,.25) {$b$};
            \node[blue] at (10.5,3.5) {$a$};
            \node[ForestGreen] at (13.5,.25) {$b$};
            \node[blue] at (13.5,3.5) {$a$};

            \node at (12,0) {$(b\, c)$};
            
           \end{tikzpicture}

        \caption{In the base diagrams shown, $a, b,$ and $c$ are all pairwise distinct. All the red branch points shown have monodromy $(a\,b)$.}
        \label{fig:bigon_lem}
    \end{figure}
\end{lemma}

\begin{proof}

    The moves exhibiting this isotopy are shown in Figure \ref{fig:lemma_proof}. From (1) to (2), one stabilization is performed. Dragging the blue $a$-labeled arc over the new branch point brings us to (3). This yields a liftable isotopy, since $a$ is different from both $b$ and $c$.  We isotope the green b-labeled arc according to the pink arrow in (3) to arrive at (4). Notice that, when passing through the cut system associated to the blue branch point, the label of the green curve changes from $b$ to $c$. That's why we can slide its $c$-labeled part over the red branch points via a liftable isotopy. An isotopy of the blue arc according to the arrow in (4) yields (5). We then eliminate the yellow-shaded bigon to reach (6), completing the proof.
   
    \begin{figure}
        \centering

        \begin{tikzpicture}
         \node[anchor=south west, inner sep=0] (fig) at (0,0) 
           {\includegraphics[width=6in]{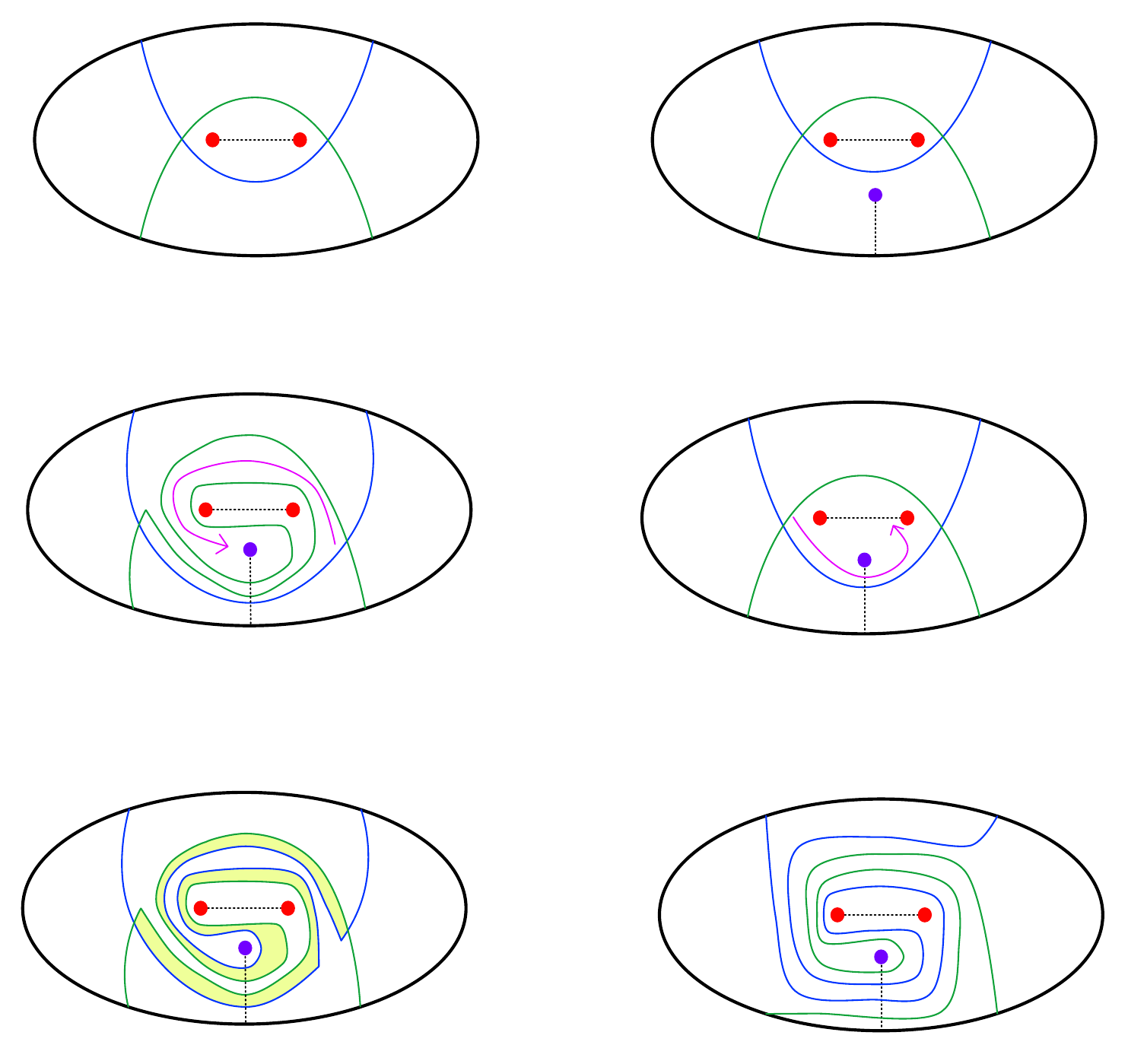}};
         
         
            \node at (.5,14.5) {(1)};
            \node[ForestGreen] at (1.6,10.9) {$b$};
            \node[blue] at (1.6,14) {$a$};
            \node at (7.5, 12.5) {\scalebox{2}{$\rightarrow$}};
            
            \node at (9,14.5) {(2)};
            \node[ForestGreen] at (10.25,10.75) {$b$};
            \node[blue] at (10.25,14.2) {$a$};
            \node at (11.75, 10.5) {$(b\, c)$};
            \node[rotate=-90] at (11.75, 9.7) {\scalebox{1.5}{$\rightarrow$}};

            \node at (9,9.5) {(3)};
            \node[ForestGreen] at (10.1,5.8) {$b$};
            \node[blue] at (10.1,9) {$a$};
            \node at (11.75, 5.5) {$(b\, c)$};
            \node at (7.5, 7.5) {\scalebox{2}{$\leftarrow$}};

            \node at (.5,9.5) {(4)};
            \node[ForestGreen] at (1.6,5.8) {$b$};
            \node[blue] at (1.6,9) {$a$};
            \node at (3.35, 5.5) {$(b\, c)$};
            \node[rotate=-90] at (3.3,4.5) {\scalebox{1.5}{$\rightarrow$}};
            
            \node at (.5,4.25) {(5)};
            \node[ForestGreen] at (1.5,.35) {$b$};
            \node[blue] at (1.5,3.55) {$a$};
            \node at (7.5, 2) {\scalebox{2}{$\rightarrow$}};
            
            \node at (9,4.25) {(6)};
            \node[ForestGreen] at (10.25,.35) {$b$};
            \node[blue] at (10.25,3.55) {$a$};
            \node[ForestGreen] at (13.5,.35) {$b$};
            \node[blue] at (13.5,3.55) {$a$};
            \node at (12,0) {$(b\, c)$};
            
           \end{tikzpicture}
           
        \caption{Proof of Lemma \ref{lem:bigon}}
        \label{fig:lemma_proof}
    \end{figure}

\end{proof}

\begin{proof}[Proof of Proposition \ref{prop:even_twist_lift}]
    Let $t_\delta$ be one of Chillingworth's generators of $\T(F)$. To prove that $t_\delta$ is the lift of a half twist, we follow the same recipe as \cite[Theorem 1]{Zuddas2007RepresentingDT}. Namely, we show that, up to stabilizations of $p$ and liftable isotopies, $p(\delta)$ bounds a disk containing exactly two branch points. We proceed by doing this for Dehn twists along the curves $\{\alpha_i\}$ and $\eta$, $\psi$, and $\chi$.
    
    By Proposition \ref{prop:chill_DT_images}, for each $i$, $p(\alpha_i)$ is an embedded loop contained entirely in $\mathcal D_p \subset \RP^2$, and which bounds at least two branch points with monodromy $(1\,2)$. Hence $p(\alpha_i)=p_1(\alpha_i)$, and it suffices to consider the restriction of $p$ to $p_1$, which is a cover of the disk $\mathcal D_p$. The arguments from \cite{Zuddas2007RepresentingDT} show that up to stabilization of $p_1$, $t_{\alpha_i}$ is the lift of a half twist. Moreover, since $p_1(\alpha_i)$ is embedded, one stabilization is enough. The two stabilizations already performed to $p_1$, wherein the branch points with monodromy $(1\, 3)$ and $(2\, 4)$ are introduced, are sufficient for isotoping $p(\alpha_i)$ so that it contains exactly two branch points. Therefore, each Dehn twist $t_{\alpha_i}$ is a lift of a half twist via the $4$-fold cover $p_1$, with no further stabilizations needed (cf. Remark \ref{rmk:ht_dt}). For more detail on how to isotope $p(\alpha_i)$, see \cite[Section 4]{Zuddas2007RepresentingDT}.

    Next, we show that Dehn twists along $\eta$, $\psi$, and $\chi$ are lifts of half twists. Below we isotope the image of each of these curves under $p$, through liftable isotopies, such that they bound disks containing exactly two branch points with the right monodromy. Figure \ref{fig:even_psi} gives such an isotopy for $\psi$, and Figure \ref{fig:even_chi} gives an isotopy for $\chi$. In both figures, the isotopies performed between each step are indicated by pink arrows. The disks bounding two branch points at the end of the isotopies are shaded yellow.

    \begin{figure}
        \centering
        
    \begin{tikzpicture}
        \node[anchor=south west, inner sep=0] (fig) at (0,0) {\includegraphics[width=6in]{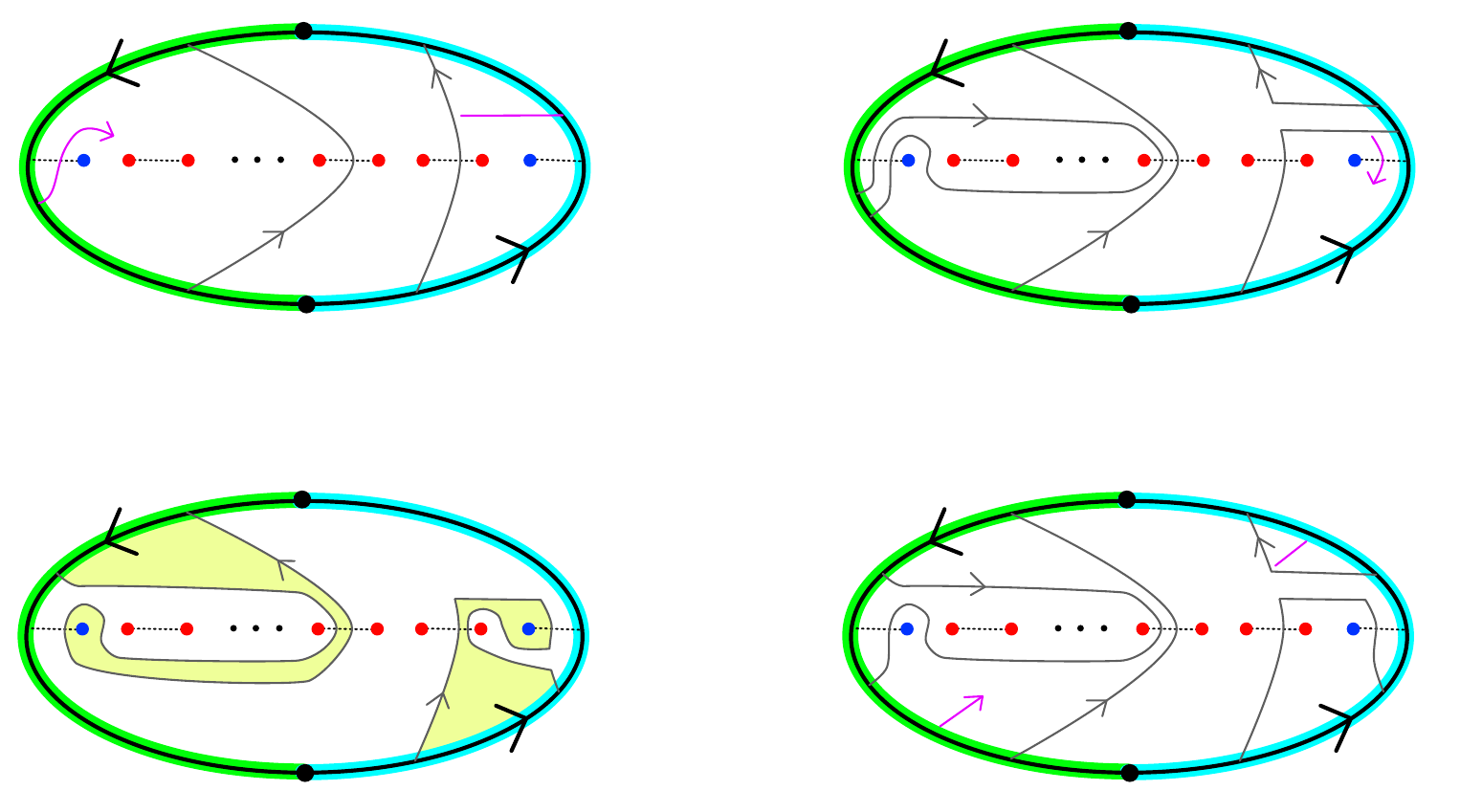}};

        \node at (0,9) {(1)};
        \node[scale=.75] at (-.25,6.75) {$(1\, 3)$};
        \node[scale=.75] at (6.5,6.75) {$(2\,4)$};
        \node[scale=2] at (7.5,6.75) {$\rightarrow$};
        \node[color=gray,scale=.75] at (2.5,6) {$1$};
        \node[color=gray,scale=.75] at (3.3,7.5) {$2$};
        \node[color=gray,scale=.75] at (4.9,6.4) {$1$};
        \node[color=gray,scale=.75] at (4.9,7.3){$2$};

        \node at (8.5,9) {(2)};
        \node[scale=.75] at (8.45,6.75) {$(1\, 3)$};
        \node[scale=.75] at (15.2,6.75) {$(2\,4)$};
        \node[color=gray,scale=.75] at (11,6) {$1$};
        \node[color=gray,scale=.75] at (13.5,6.4) {$1$};
        \node[color=gray,scale=.75] at (10,6.3) {$3$};
        \node[rotate=-90,scale=2] at (11.8,4.3) {$\rightarrow$};

        \node at (8.5,3.5) {(3)};
        \node[scale=.75] at (8.4,1.9) {$(1\,3)$};
        \node[scale=.75] at (15.2,1.9) {$(2\,4)$};
        \node[color=gray,scale=.75] at (11,1) {$1$};
        \node[color=gray,scale=.75] at (13.5,1.5) {$1$};
        \node[color=gray,scale=.75] at (10,1.4) {$3$};
        \node[scale=2] at (7.4,1.9) {$\leftarrow$};

        \node at (0,3.5) {(4)};
        \node[scale=.75] at (-.25,1.9) {$(1\,3)$};
        \node[scale=.75] at (6.5,1.9) {$(2\, 4)$};
        \node[scale=.75,color=gray] at (4.3,.8) {$1$};
        \node[scale=.75,color=gray] at (1.4,1.2) {$1$};

            
    \end{tikzpicture}
        \caption{Isotoping the image of $\psi$ in the even genus case.}
        \label{fig:even_psi}
    \end{figure}

    \begin{figure}
        \centering
            \begin{tikzpicture}
            \node[anchor=south west, inner sep=0] (fig) at (0,0) {\includegraphics[width=6in]{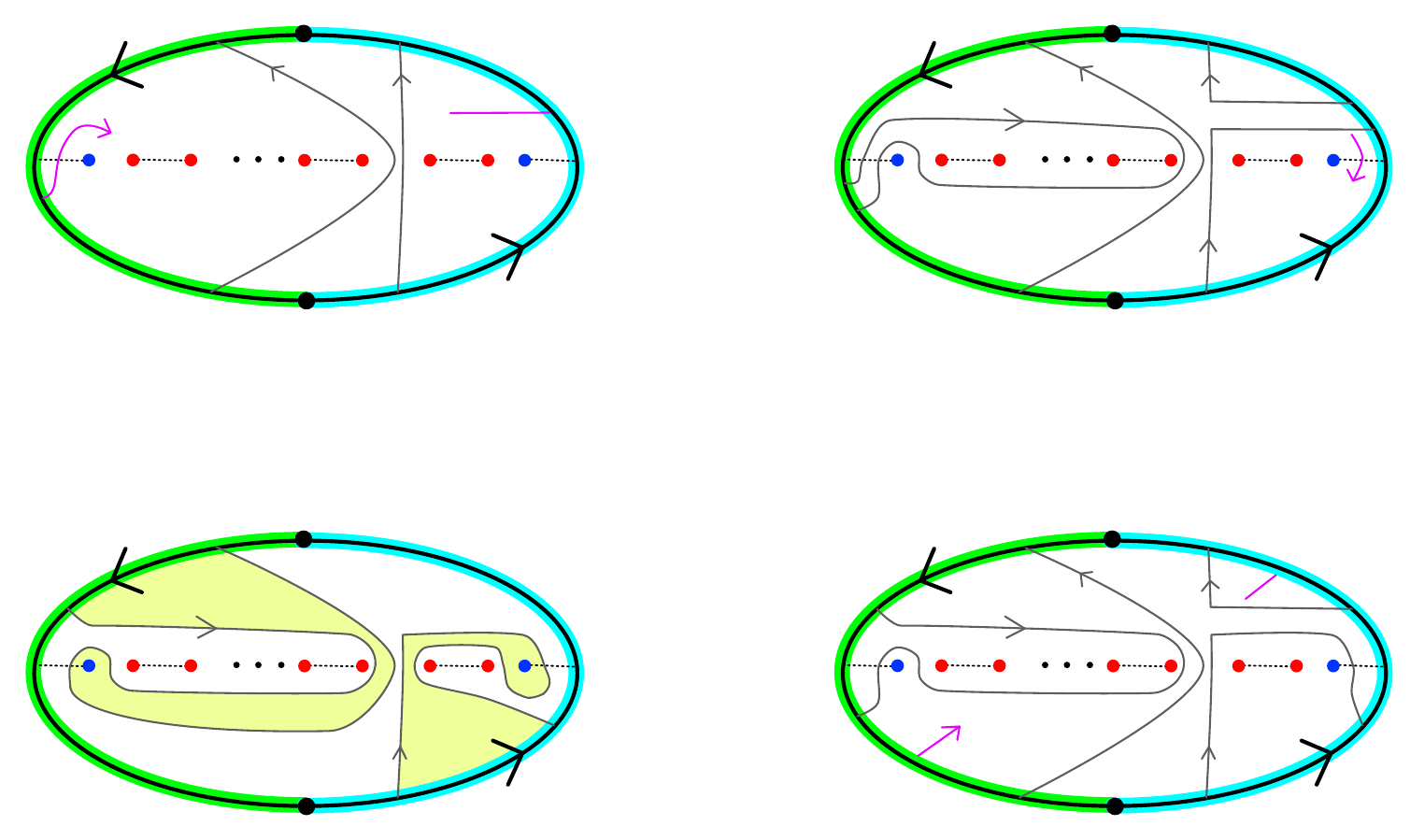}};
    
            \node at (0,9) {(1)};
            \node[scale=.75] at (-.1,7.3) {$(1\, 3)$};
            \node[scale=.75] at (6.7,7.3) {$(2\,4)$};
            \node[scale=2] at (7.6,7.3) {$\rightarrow$};
            \node[color=gray,scale=.75] at (3,6.5) {$1$};
            \node[color=gray,scale=.75] at (4.55,6.5) {$2$};

            \node at (8.5,9) {(2)};
            \node[scale=.75] at (8.65,7.3) {$(1\, 3)$};
            \node[scale=.75] at (15.4,7.3) {$(2\,4)$};
            \node[color=gray,scale=.75] at (11.6,6.5) {$1$};
            \node[color=gray,scale=.75] at (13.4,6.5) {$2$};
            \node[color=gray,scale=.75] at (10.5,6.8) {$3$};
            \node[rotate=-90,scale=2] at (12,4.5) {$\rightarrow$};
    
            \node at (8.5,3.5) {(3)};
            \node[scale=.75] at (8.65,1.9) {$(1\,3)$};
            \node[scale=.75] at (15.4,1.9) {$(2\,4)$};
            \node[color=gray,scale=.75] at (11.6,1) {$1$};
            \node[color=gray,scale=.75] at (13.4,1) {$2$};
            \node[color=gray,scale=.75] at (10,1.4) {$3$};
            \node[scale=2] at (7.7,1.9) {$\leftarrow$};
    
            \node at (0,3.5) {(4)};
            \node[scale=.75] at (-.1,1.9) {$(1\,3)$};
            \node[scale=.75] at (6.7,1.9) {$(2\, 4)$};
            \node[color=gray,scale=.75] at (4.1,.8) {$2$};
            \node[scale=.75,color=gray] at (1.4,1.1) {$1$};

                
        \end{tikzpicture}
        
        \caption{Isotoping the image of $\chi$ in the even genus case.}
        \label{fig:even_chi}
    \end{figure}

    Lastly, we show that $t_\eta$ is the lift of a half twist. The isotopies for this are shown in Figure \ref{fig:even_eta}. In this figure, we apply Lemma \ref{lem:bigon}, and then stabilize an additional time as in Remark \ref{rmk:stabilizations} to arrive at (2). Then we isotope the curve as indicated by the arrow to arrive at (3). Once again, the disk bounding two branch points is shaded yellow.

    \begin{figure}
        \centering
        \begin{tikzpicture}
                \node[anchor=south west, inner sep=0] (fig) at (0,0) {\includegraphics[width=3in]{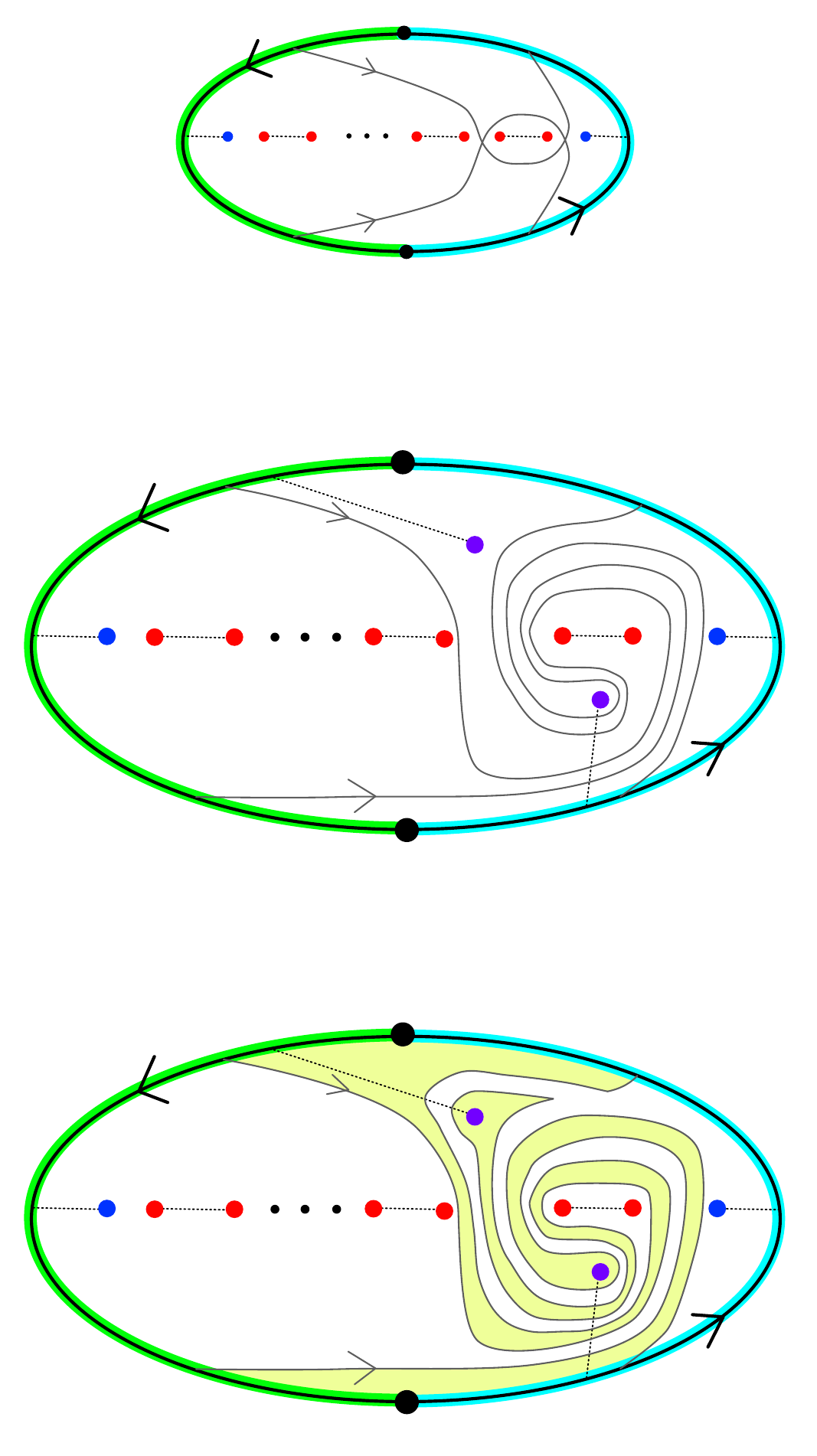}};
    
                \node at (0,14) {(1)};
                \node[scale=.75] at (1.3,12.3) {$(1\, 3)$};
                \node[scale=.75] at (6.3,12.3) {$(2\,4)$};
                \node[color=gray,scale=.75] at (3.8,11.8) {$1$};
                \node[color=gray, scale=.75] at (4.3,12.8) {$2$};
                \node[scale=2,rotate=-90] at (3.7,10.3) {$\rightarrow$};
    
                \node at (0,9.3) {(2)};
                \node[scale=.75] at (-.2,7.6) {$(1\,3)$};
                \node[scale=.75] at (7.7,7.6) {$(2\, 4)$};
                \node[scale=.75] at (5.6,5.7) {$(1\, 5)$};
                \node[scale=.75] at (2.4,9.5) {$(2\,6)$};
                \node[scale=.75,color=gray] at (2.5,6.3) {$1$};
                \node[scale=.75,color=gray] at (2.6,8.7) {$2$};
                \node[scale=2,rotate=-90] at (3.7,4.9) {$\rightarrow$};
    
                \node at (0,4) {(3)};
                \node[scale=.75] at (-.2,2.35) {$(1\,3)$};
                \node[scale=.75] at (7.7,2.35) {$(2\, 4)$};
                \node[scale=.75] at (5.6,.4) {$(1\, 5)$};
                \node[scale=.75] at (2.4,4.2) {$(2\,6)$};
                \node[scale=.75,color=gray] at (2.5,1) {$1$};
                \node[scale=.75,color=gray] at (2.6,3.4) {$2$};
                
            \end{tikzpicture}
        \caption{Isotoping the image of $\eta$ in the even genus case.}
        \label{fig:even_eta}
    \end{figure}

    This completes the proof that each generator of $\T(F)$ is the lift of a half twist. At least two stabilizations are needed, so the total degree of $p$ is at least six by the end of this proof.
\end{proof}

\subsection{Case 2: $g(F)=k$ is odd}\label{sec:k_odd}

Now we begin with the degree-three cover $q\colon F\rightarrow \RP^2$ from Example \ref{eg:diagram3}(2), and show that, up to a single pair of stabilizations, $\lambda_q$ surjects onto $\T(F)$. Chillingworth \cite{Chillingworth_1969} shows that in the odd case, $\T(F)$ is generated by Dehn twists along the following 2-sided simple loops:
\begin{itemize}
    \item[-] the collection of loops given in Theorem \ref{thm:chillingworth_69} (ii), shown in Figure \ref{fig:chill_generators}. We will re-label these curves and denote the entire collection by $\{\alpha_i\}$. Note that each $\alpha_i$ is contained in the interior of $\Sigma_g^1= F \setminus \Int D^2$;

    \item[-] the curves $\eta$ and $\psi$, which we define below.
\end{itemize}

Figure \ref{fig:ep} shows the intersection of $\eta$ and $\psi$ with $\Sigma_g^1=F \setminus \Int D^2$. To complete them to simple loops, we identify antipodal points of their intersections with $\partial \Sigma_g^1$. Their images on $\mathcal D_q$ under $q$ are shown on the left-hand sides of Figures \ref{fig:odd_psi} and \ref{fig:odd_eta}, respectively.
   \begin{figure}
    \centering
     \begin{overpic}[width=6in]{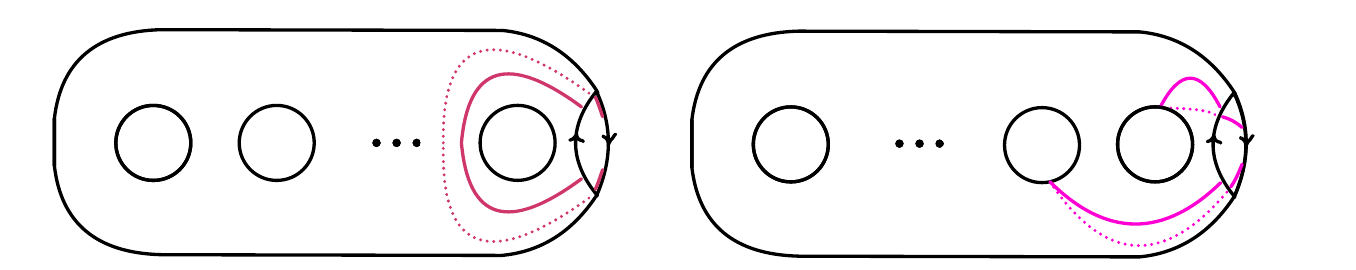}
        \put(31,15){\textcolor{purple}{$\eta$}}
        \put(84,14){\textcolor{RubineRed}{$\psi$}}
    \end{overpic}
    \caption{The curves $\eta$ and $\psi$ in the odd genus case.} \label{fig:ep}

\end{figure}

\begin{proposition}\label{prop:odd_twist_lift}
    Up to stabilization, each generator of $\T(F)$ is the lift of a half twist with respect to $q$. In particular, it is enough to stabilize $q$ up to degree $5$.
\end{proposition}

\begin{proof}
    Just like the even case, the methods in \cite{Zuddas2007RepresentingDT} may be applied to show that modulo stabilization, a Dehn twist around each $\alpha_i$ is a lift of a half twist. Similar to Proposition \ref{prop:even_twist_lift}, the stabilizations already performed to $q_1$ are sufficient to ensure that each $t_{\alpha_i}$ is the lift of a half twist, so the degree of $q$ is still three up to this point. Next we show that $t_\eta$ and $t_\psi$ are both lifts of half twists. This amounts to showing that their images on $\RP^2$ are related through stabilizations and liftable isotopies to curves which bound disks containing exactly two branch points. All stabilizations must be performed in accordance with Remark \ref{rmk:stabilizations}.

    No stabilizations are needed to show that $t_\psi$ is a lift of a half twist. On the other hand, we use one pair of stabilizations to exhibit this same property for $t_\eta$. Figure \ref{fig:odd_psi} (resp. Figure \ref{fig:odd_eta}) shows the image of $\psi$ (resp. $\eta$) before and after the stabilizations and isotopies employed. The isotopies and stabilizations performed are similar to those in the even genus case, which are described in more detail in the proof of Proposition \ref{prop:even_twist_lift}. Hence each generator of $\T(F)$ is the lift of a half twist. At least two stabilizations are needed, so the degree of $q$ by the end of this proof is at least five.\qedhere

    \begin{figure}
        \centering
            \begin{tikzpicture}
                 \node[anchor=south west, inner sep=0] (fig) at (0,0) 
                   {\includegraphics[width=6in]{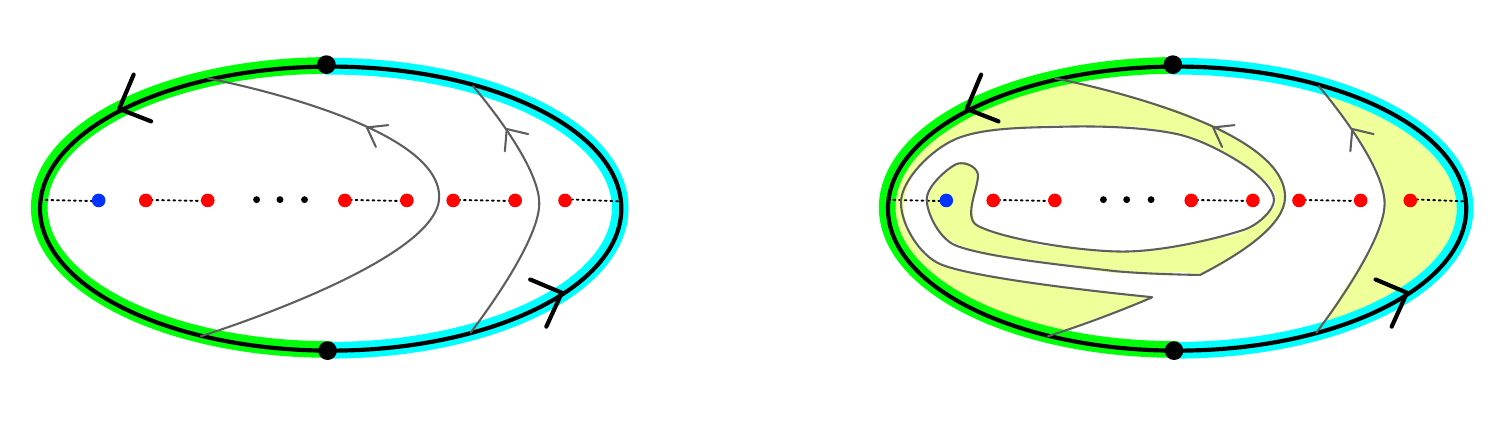}};
                \node at (-.1, 2.35) {\scalebox{.75}{$(1\, 3)$}};
                \node at (8.5, 2.35) {\scalebox{.75}{$(1\, 3)$}};
                \node at (7.25, 2.4) {\scalebox{2}{$\rightarrow$}};
            
                \node at (2.5,1.4) {\scalebox{.75}{\textcolor{gray}{$1$}}};
                \node at (5.6,2.8) {\scalebox{.75}{\textcolor{gray}{$2$}}};
               \end{tikzpicture}                        
        \caption{Isotoping the image of $\psi$ in the odd genus case.}
        \label{fig:odd_psi}
    \end{figure}
           
    \begin{figure}
        \centering
        \begin{tikzpicture}
             \node[anchor=south west, inner sep=0] (fig) at (0,0) 
               {\includegraphics[width=6in]{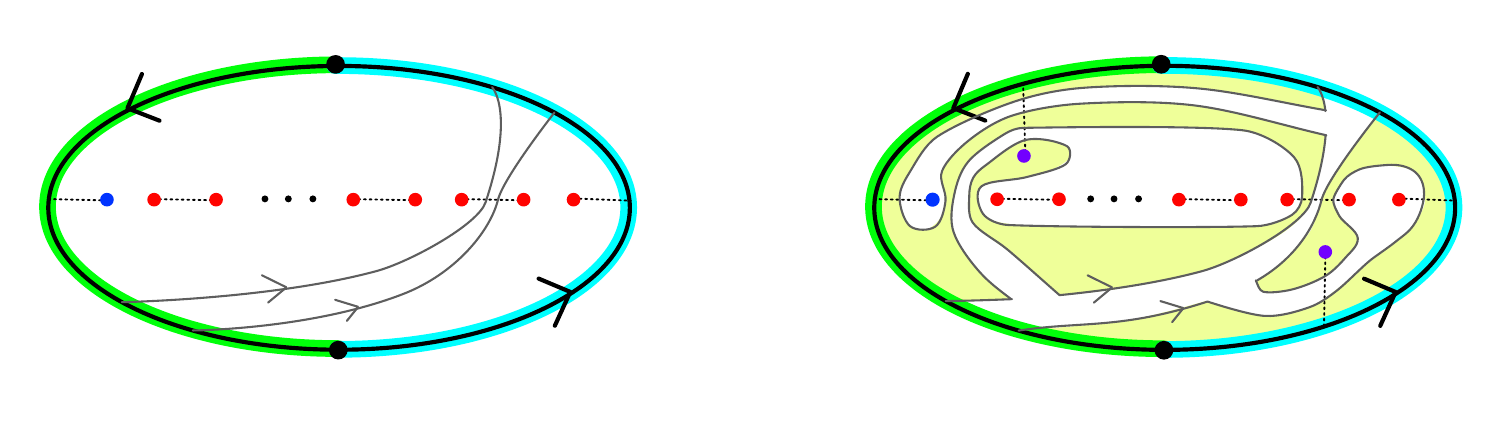}};
            \node at (0, 2.4) {\scalebox{.75}{$(1\, 3)$}};
            \node at (8.4, 2.4) {\scalebox{.75}{$(1\, 3)$}};
            \node at (13.5,.75) {\scalebox{.75}{$(1\,4)$}};
            \node at (10.5,4) {\scalebox{.75}{$(1\,5)$}};
            \node at (7.25, 2.4) {\scalebox{2}{$\rightarrow$}};
        
            \node at (4.5,1.4) {\scalebox{.75}{\textcolor{gray}{$1$}}};
            \node at (3.5,1.75) {\scalebox{.75}{\textcolor{gray}{$1$}}};
            \node at (4.9,3.2) {\scalebox{.75}{\textcolor{gray}{$2$}}};
            \node at (5.6,2.9) {\scalebox{.75}{\textcolor{gray}{$2$}}};
           \end{tikzpicture}
        
        \caption{Isotoping the image of $\eta$ in the odd genus case.}
        \label{fig:odd_eta}
    \end{figure} 
    \end{proof}

Propositions \ref{prop:even_twist_lift} and \ref{prop:odd_twist_lift} together complete the proof of Theorem \ref{thm:chil}. As a direct consequence of Theorem \ref{thm:chil} and Lemma \ref{lem:lifting}, we get the following corollary.
\begin{corollary}\label{cor:lift}
    Let $F$ be a closed connected non-orientable surface and $p \colon F \rightarrow \RP^2$ a simple branched cover of degree $d \geq 5$. Then the image of the lifting homomorphism
    \[\lambda_p \colon \L_p \longrightarrow  \M(F) \]
    contains the twist subgroup $\T(F)$.
\end{corollary}
\begin{proof}
    By Theorem \ref{thm:chil} there is a simple $d$-fold branched cover $p'\colon F \to \RP^2$ with $\T(F)\subseteq \mathrm{im}\,\lambda_{p'}$. By Theorem \ref{thm:gabai_kazez}, $p$ and $p'$ are equivalent, and by Lemma \ref{lem:lifting} the images of $\lambda_p$ and $\lambda_{p'}$ differ by conjugation by a fixed element of $\M(F)$. Since $\T(F)$ is normal in $\M(F)$, the conclusion follows.
\end{proof}

\subsection{The lifting homomorphism is surjective}
In light of Corollary \ref{cor:lift}, and since $\T(F)$ has index two in $\M(F)$, in order to complete the proof of Theorem \ref{thm:1} it suffices to show that there exists a simple branched cover $p \colon F \rightarrow \mathbb{RP}^2$ for which at least one element in $\M(F) \setminus \T(F)$ is in the image of $\lambda_p$. The conclusion will then follow from Theorem \ref{thm:gabai_kazez}. The aim of this subsection is to build such cover $p$ for every allowable degree $d \geq 5$. We will use the following result from Stukow to verify that a given homeomorphism is not in the twist subgroup.

\begin{proposition}\cite{Stukow2007}\label{prop:stukow}
    Let $F$ be a closed non-orientable surface with $g(F)\geq 2$ and let $f\colon F\to F$ be a homeomorphism. Then $[f]\in \T(F)$ if and only if $f_*\colon H_1(F;\R)\rightarrow H_1(F;\R)$ has determinant $1$.
\end{proposition} In other words, taking the determinant of $f_*$ gives an isomorphism between the quotient $\M(F)/\T(F)$ and $\{\pm 1\}$.


\begin{lemma}\label{prop:evengenus_Y}
    Let $F$ be a closed connected non-orientable surface of even genus $g(F) \geq 2$. Then, for every even integer $d \geq 2$, there is a simple $d$-fold branched cover $p \colon F \rightarrow \mathbb{RP}^2$ such that the image of the lifting homomorphism $\lambda_p$ contains an element in $\M(F) \setminus \T(F)$.
\end{lemma}
\begin{proof}
    Let $p \colon F \rightarrow \mathbb{RP}^2$ be the simple $d$-fold branched cover introduced in Example \ref{eg:diagram3}(1). In particular, $p$ is obtained by gluing together 
    \[p_1 \colon F\setminus \text{Int }(\nu(\gamma)) \longrightarrow D^2 \quad \text{and} \quad p_2 \colon \nu(\gamma) \longrightarrow Mb,\]
    where $\gamma \subset F$ is a simple loop dual to $w_1(F)$, $p_1$ is a stabilization of the $2$-fold cover induced by the hyperelliptic involution, and $p_2$ is a regular $d$-fold cover.

    Now consider the self-diffeomorphism $b \colon \mathbb{RP}^2 \setminus B_p \rightarrow \mathbb{RP}^2 \setminus B_p$ defined as follows. Write $\mathbb{RP}^2=D^2 \cup Mb$, where $Mb \cong [-1,1]^2/\sim$ and $(-1,t)\sim (1,-t)$. On $Mb$, let $b(s,t)=(s,-t)$. This map is well defined on $Mb$ and will be called a \textsl{reflection of $Mb$ along its core}. Note that $b|_{\partial Mb}$ is the antipodal map. Up to an ambient isotopy of the branch set, we can assume without loss of generality that all the branch points of $p_1$ whose monodromy is not $(1\, 2)$ are contained in a small collar $C$ of $\partial D^2$. Moreover, within $C$, we assume that the branch points with monodromy $(1\,3), \dots , (1\,d-1)$ are mapped to the branch points with monodromy $(2\,4), \dots, (2\,d)$ via the antipodal map. We then extend $b$ to all of $D^2$ using the antipodal map times the identity of the interval on $C\cong S^1 \times I$, an isotopy from the antipodal map of $S^1$ to the identity map of $S^1$ in a collar of $\partial (D^2 \setminus C)$ disjoint from the points in $B_p$ with monodromy $(1\,2)$, and then the identity map on the remaining part of $D^2$. Now, $b$ defines a self-diffeomorphism of $\RP^2$, which is induced by some braid in $\B_n(\RP^2)$. The lifting criterion in Proposition \ref{prop:liftcrit} implies that $b$ is liftable and we denote by $\phi$ its lift, which is unique up to deck transformations of $p$. 
    
    We now claim that $[\phi] \in \M(F)\setminus \T(F)$. Indeed, $b|_{Mb}$ lifts to a reflection of the cylinder $\nu(\gamma)$ along its meridional circle. On the other hand, the restriction $b|_{D^2}$ lifts to a map on $F\setminus \text{Int }\nu(\gamma)$ induced by a hyperelliptic involution. This can be seen as follows. When $d=2$, the restriction of $p$ to $F \setminus \text{Int } \nu(\gamma)$ is regular and the unique non-trivial deck transformation is given by a $\pi$-rotation with respect to a symmetry axis. This implies that the lift of $b|_{D^2}$ via $p$ is either the identity or this non-trivial deck transformation. Since the lift of $b|_{\partial Mb}$ exchanges the two components of $\partial \nu(\gamma)$, the lift of $b|_{D^2}$ cannot be the identity. For the case when $d > 2$, it is enough to notice that stabilizing $p|_{F \setminus \text{Int } \nu(\gamma)}$ over $D^2$ does not change the isotopy type of the lift of $b|_{D^2}$. The conclusion then follows from Proposition \ref{prop:stukow}, since it can be easily checked that $\phi_*\colon H_1(F;\R)\rightarrow H_1(F;\R)$ has determinant $-1$.
\end{proof}

Next, we extend this construction to surfaces of odd genus.

\begin{lemma}\label{prop:oddgenus_Y}
    Let $F$ be a closed connected non-orientable surface of odd genus $g(F)\geq 3$. Then, for every odd integer $d \geq 3$, there is a simple $d$-fold branched cover $p \colon F \rightarrow \mathbb{RP}^2$ such that the image of the lifting homomorphism $\lambda_p$ contains an element in $\M(F) \setminus \T(F)$.
\end{lemma}

\begin{proof}

    Since $g(F)\geq 3$ is odd, we can write $F=F'\# \mathbb{RP}^2$, where $F'$ is a closed non-orientable surface of even genus $g(F)-1\geq 2$. Given $d\geq 3$ odd, let $p' \colon F' \rightarrow \mathbb{RP}^2$ be the $(d-1)$-fold cover constructed in Lemma \ref{prop:evengenus_Y}. We now take the disjoint union of $p'$ with the identity of $\mathbb{RP}^2$ to get a simple $d$-fold branched cover
    \[p''\colon F' \sqcup \mathbb{RP}^2 \longrightarrow \mathbb{RP}^2,\]
    where $\mathbb{RP}^2$ is contained in the first sheet.
    
    We now consider the simple $d$-fold branched cover
        \[q \colon A\cup( \sqcup_{d-3} D^2) \longrightarrow D^2\]
     whose branch set consists of three branch points with monodromy $(1\,2)$, $(1\,2)$ and $(2\,3)$ respectively. Here $A$ denotes an annulus. Let $S$ be a 2-disk in $\RP^2$ containing exactly one branch point with monodromy $(2\, 3)$ in its interior. Remove $S$ from $\mathbb{RP}^2$, and glue the restriction of $p''$ to the complement of $S$ to $q$. In this way, we get a simple $d$-fold branched cover 
    \[p \colon F \longrightarrow \mathbb{RP}^2.\] It's helpful to observe that $(p'')^{-1}(S)$ contains one disk $\widetilde S_1\subset \RP^2$ and $d-2$ disks in $F'$, one of which is fixed setwise by the hyperelliptic involution on $F'$ and will be denoted $\widetilde S_2$. The total space of $p$ is obtained by connect-summing $F'$ and $\RP^2$ along $\widetilde S_1$ and $\widetilde S_2$, and leaving the remainder of $p''^{-1}(S)$ unchanged.
    
    Let $b'$ be the self-diffeomorphism of $\RP^2$ defined as in the proof of Lemma \ref{prop:evengenus_Y}. In particular, the lift of $b'$ via $p'$ is an element in $\M(F')\setminus \T(F')$. Let $b$ be the self-diffeomorphism defined by gluing the restriction of $b'$ over the complement of the disk $S$ with the identity map. By the lifting criterion in Proposition \ref{prop:liftcrit}, $b$ is liftable; let $\phi$ be its lift. To conclude the proof, we argue that $[\phi]\in \M(F)\setminus \T(F)$. The part of $\phi$ supported in $\RP^2\setminus \widetilde S_1$ is necessarily isotopic to the identity rel. boundary, since $\M(Mb,\partial Mb)$ is trivial \cite{epstein}. The part of $\phi$ supported in $F'\setminus \widetilde S_2$ is isotopic to the lift constructed in Lemma \ref{prop:evengenus_Y}, restricted to $F'\setminus \widetilde S_2$. As in the even genus case, one can then argue that the induced map $\phi_*$ has determinant $-1$, and the conclusion again follows from Proposition \ref{prop:stukow}.
\end{proof}

\begin{proof}[Proof of Theorem \ref{thm:1}]
    If $F\cong\RP^2$ then $\M(F)$ is trivial and there is nothing to prove, so assume $g(F)\geq 2$. By Lemmas \ref{prop:evengenus_Y} and \ref{prop:oddgenus_Y}, for every integer $d \geq 2$ with the same parity as $\chi(F)$, there exists a simple $d$-fold branched cover
    \[p \colon F \longrightarrow \mathbb{RP}^2\] such that the image of $\lambda_p$ contains at least one element in $\M(F) \setminus \T(F)$. By Corollary \ref{cor:lift}, when $d \geq 5$ the image of $\lambda_p$ contains the whole Dehn twist subgroup as well. Since $\T(F)$ has index two in $\M(F)$, it follows that if $d \geq 5$ then $\lambda_p$ is surjective. If $q \colon F \rightarrow \mathbb{RP}^2$ is any other simple $d$-fold branched cover of the same degree, then $q$ is equivalent to $p$ by Theorem \ref{thm:gabai_kazez}. The conclusion follows from Lemma \ref{lem:lifting}.
\end{proof}

\begin{remark}\label{rmk:or}
A slight extension of Theorem \ref{thm:1} to closed \textsl{orientable} surfaces is the following. If $F$ is a closed connected orientable surface, then for any simple $d$-fold branched cover $p \colon F \rightarrow \RP^2$, $d$ is necessarily even and there exists a commutative diagram
\begin{center}
\begin{tikzcd}
F \arrow[r, "\widetilde p"] \arrow[rd, "p"'] & S^2 \arrow[d, "u"] \\
                                             & \RP^2,             
\end{tikzcd}
\end{center}
where $\widetilde p \colon F \rightarrow S^2$ is a simple $d/2$-fold branched cover and $u \colon S^2 \rightarrow \RP^2$ is the universal covering map of $\RP^2$. This implies that a self-diffeomorphism of $\RP^2$ is liftable with respect to $p$ if and only if its lift with respect to $u$ is liftable with respect to $\widetilde p$.
\end{remark}

\section{Dehn twists as lifts of half twists}\label{sec:9}
In this section, we establish some technical results that will be needed in the proof of Theorem \ref{thm:3}. In particular, given a closed non-orientable surface $F$ and a simple $2$-sided loop $\gamma \subset F$, we show that we can always find a simple branched cover $p \colon F \rightarrow \mathbb{RP}^2$ such that the Dehn twist $t_{\gamma}$ is isotopic to the lift of a half twist. We also give bounds on the degree of $p$. There are three cases to consider: $\gamma$ separates $F$ and bounds an orientable surface, $\gamma$ separates $F$ into two non-orientable surfaces, and $\gamma$ is non-separating. We treat each of these three cases in the following three propositions.

\begin{proposition}\label{prop:sep1}
    Let $F$ be a closed connected non-orientable surface and $\gamma \subset F$ a simple $2$-sided loop such that 
    \[F=\Sigma^1_{g_1} \cup_{\gamma} N^1_{g_2},\]
    where $\Sigma^1_{g_1}$ (resp. $N^1_{g_2}$) is an orientable (resp. non-orientable) surface of genus $g_1$ (resp. $g_2$) and one boundary component. Then, for every integer $d \geq 6$ such that $d \equiv \chi(F) \pmod{2}$, there exists a simple $d$-fold branched cover
    \[p \colon F \longrightarrow \mathbb{RP}^2\]
    such that the Dehn twist $t_{\gamma}$ is the lift of a half twist around a simple arc $a \subset \mathbb{RP}^2$.
\end{proposition}
\begin{proof}
Consider the closed surfaces
\[\Sigma_{g_1}=\Sigma^1_{g_1}\cup D^2 \quad \text{and} \quad N_{g_2}=N^1_{g_2}\cup D^2.\]

Notice that
\[\chi(N_{g_2})\equiv \chi(F) \pmod{2}.\]

There exist simple branched covers
\[p_1 \colon \Sigma_{g_1} \xrightarrow{4:1} \mathbb{RP}^2 \quad \text{and} \quad p_2\colon N_{g_2} \xrightarrow{d_2:1} \mathbb{RP}^2\]
where $d_2=2$ if $\chi(F)$ is even and $d_2=3$ if $\chi(F)$ is odd (see for example \cite[Lemma 1.7]{bais2025branched}). Then the branched cover
\[p_1 \sqcup p_2 \colon \Sigma_{g_1} \sqcup N_{g_2} \xrightarrow{d:1} \mathbb{RP}^2\]
has degree $d=6$ if $\chi(F)$ is even and $d=7$ if $\chi(F)$ is odd.

We can assume without loss of generality that $B_{p_1}$ and $B_{p_2}$ are disjoint, so that $p_1 \sqcup p_2$ is simple and $B_{p_1\sqcup p_2}=B_{p_1}\sqcup B_{p_2}$. We number the sheets of $p_1 \sqcup p_2$ so that the first $4$ sheets belong to $p_1$. Now let $D^2 \subset \mathbb{RP}^2$ be an embedded $2$-disk which is disjoint from $B_{p_1 \sqcup p_2}$ and let
\[q \colon \sqcup_{d-2} D^2 \sqcup S^1 \times I \xrightarrow{d:1} D^2\]
be a simple $d$-fold branched cover with two disjoint branch points whose monodromy is $(4 \; 5)$. In particular, $q$ is a homeomorphism over the disjoint union $\sqcup_{d-2} D^2$, while its restriction to $S^1 \times I$ is $2$-fold. We then glue the restriction of $p_1 \sqcup p_2$ over $\mathbb{RP}^2\setminus \text{Int } D^2$ with $q$ to get a simple $d$-fold branched cover
\[p \colon  F \cong \Sigma_{g_1}\# N_{g_2} \xrightarrow{d:1} \mathbb{RP}^2.\]
The conclusion follows from noticing that the Dehn twist $t_{\gamma}$ is the lift of a half twist around a simple arc connecting the two branch points of $q$ in $D^2$ (see Remark \ref{rmk:ht_dt}). This proves the statement for $d=6,7$. One can then arbitrarily increase the degree using Proposition \ref{prop:stab}.

\end{proof}

\begin{proposition}\label{prop:sep2}
    Let $F$ be a closed connected non-orientable surface and $\gamma \subset F$ a simple $2$-sided loop such that
    \[F=N^1_{g_1}\cup_{\gamma}N^1_{g_2},\]
    where $N^1_{g_i}$ is a non-orientable surface of genus $g_i$ and one boundary component for $i=1,2$. Then, for every $d \geq 5$ such that $d \equiv \chi(F) \pmod{2}$, there exists a simple $d$-fold branched cover
    \[p \colon F \longrightarrow \mathbb{RP}^2\]
    such that the Dehn twist $t_{\gamma}$ is the lift of a half twist around a simple arc $a \subset \mathbb{RP}^2$.
\end{proposition}

\begin{proof}
    Consider the closed surface
\[N_{g_i}=N^1_{g_i}\cup D^2\] for $i=1,2$. There exist simple branched covers
\[p_1 \colon N_{g_1} \xrightarrow{d_1:1} \mathbb{RP}^2 \quad \text{and} \quad p_2\colon N_{g_2} \xrightarrow{d_2:1} \mathbb{RP}^2\]
where $d_i=2$ if $\chi(N_{g_i})$ is even and $d_i=3$ if $\chi(N_{g_i})$ is odd for $i=1,2$ (see for example \cite[Lemma 1.7]{bais2025branched}). Then the branched cover
\[p_1 \sqcup p_2 \colon N_{g_1} \sqcup N_{g_2} \xrightarrow{d:1} \mathbb{RP}^2\]
has degree $d=4$ if both $g_1$ and $g_2$ are even, $d=5$ if $g_1$ is even and $g_2$ is odd or vice versa, and $d=6$ if both $g_1$ and $g_2$ are odd. The remaining part of the proof follows that of Proposition \ref{prop:sep1}. At the end, we get a simple $d$-fold branched cover
\[p \colon  F \cong N_{g_1}\# N_{g_2} \xrightarrow{d:1} \mathbb{RP}^2\]
where the Dehn twist $t_{\gamma}$ is the lift of a half twist and the degree can be increased using Proposition \ref{prop:stab}.
\end{proof}

\begin{proposition}\label{prop:nonsep}
Let $F$ be a closed connected non-orientable surface and $\gamma \subset F$ a simple $2$-sided non-separating loop. Then for every $d \geq 5$ such that $d \equiv \chi(F) \pmod{2}$, there exists a simple $d$-fold branched cover $p \colon F \rightarrow \mathbb{RP}^2$ such that the Dehn twist $t_{\gamma}$ is the lift of a half twist around a simple arc $a \subset \mathbb{RP}^2$.    
\end{proposition}
\begin{proof}
The equivalence class of $\gamma\subset F$ is determined by the homeomorphism type of its complement. It follows that, when $\gamma$ is $2$-sided and non-separating, there is one such equivalence class if $\chi(F)$ is odd or if $g(F)=2$, and two equivalence classes if $g(F)$ is even and $g(F)\geq 4$, depending on whether or not $[\gamma]=\PD(w_1(F))$. In all cases, the fact that $t_\gamma$ is the lift of a half twist via some simple branched cover for any allowable degree $d \geq 5$ was shown to be true in the proof of Theorem \ref{thm:chil}.
\end{proof}

Combining these results with Theorem \ref{thm:1}, we deduce the following.

\begin{corollary}\label{thm:lift}
Let $N_g$ be a closed connected non-orientable surface of genus $g$ and let $p \colon N_g \rightarrow \mathbb{RP}^2$ be a simple $d$-fold branched cover, where $d \geq 6$. Then any simple $2$-sided loop $\gamma \subset N_g$ is isotopic to a simple curve $\gamma'$ such that the Dehn twist $t_{\gamma'}$ is the lift of a half twist around a simple arc $a \subset \mathbb{RP}^2$.

\end{corollary}

\begin{proof}
    By Propositions \ref{prop:sep1}, \ref{prop:sep2} and \ref{prop:nonsep}, there exists a simple $d$-fold branched cover $q \colon N_g \rightarrow \mathbb{RP}^2$ such that $t_{\gamma}$ is (up to isotopy) the lift of a half twist $h_a$. By Theorem \ref{thm:gabai_kazez}, $p$ is equivalent to such cover. In particular, there exists a commutative diagram
    \begin{center}
\begin{tikzcd}
N_g \arrow[d, "q"'] \arrow[r, "\widetilde f"] & N_g \arrow[d, "p"] \\
\mathbb{RP}^2 \arrow[r, "f"']                 & \mathbb{RP}^2,     
\end{tikzcd}
    \end{center}
where $f$ and $\widetilde f$ are diffeomorphisms. Moreover, there is a commutative square
    \begin{center}
        %
\begin{tikzcd}
N_g \arrow[d, "q"'] \arrow[r, "t_{\gamma}"] & N_g \arrow[d, "q"] \\
\mathbb{RP}^2 \arrow[r, "h_a"']             & \mathbb{RP}^2.     
\end{tikzcd}
    \end{center}
    The commutative diagrams imply that
    \[\lambda_p([f^{-1}\circ h_a \circ f])=\lambda_p(h_{f^{-1}(a)})=[\widetilde f^{-1} \circ t_{\gamma} \circ \widetilde f].\]
    This means that
    \begin{equation*}[t_{\gamma}]=\widetilde f \circ \lambda_p(h_{f^{-1}(a)})\circ \widetilde f^{-1}\label{eq:s5_1}\end{equation*}
    so $t_{\gamma}$ is conjugate to the lift of a half twist. Now, Theorem \ref{thm:1} implies that there exists a braid $b$ over the branch points of $p$ such that $\lambda_p(b)= \widetilde f$. In particular, we have that
    \[t_{\gamma}=\lambda_p(b)\circ \lambda_p(h_{f^{-1}(a)})\circ \lambda_p(b)^{-1}=\lambda_p(h_{b \circ f^{-1}(a)}),\]
    i.e. $t_{\gamma}$ is the lift of a half twist.
\end{proof}

\section{Braided surfaces}\label{sec:6}

\begin{definition}[Braided surface]\label{def:brsurf} Let $B$ and $C$ be two connected compact surfaces. A properly embedded compact surface $S\subset B \times C$ is \textsl{braided over $B$} if the restriction of the canonical projection
\[\pi_B|_S \colon S \longrightarrow B\]
is a simple branched cover.

Notice that $\pi_B|_S$ has a finite number of critical values $b_1, \dots, b_n \in \text{Int } B$ and over each critical value $b_i$ there is exactly one critical point $s_i$, because $\pi_B|_S$ is simple. We call $s_1, \dots, s_n \in S$ the \textsl{twist points} of the braided surface. More detailed exposition about braided surfaces and their relation to Lefschetz fibrations can be found in \cite{Apostolakis_2013, LoiPiergallini}.    
\end{definition}
A braided surface $ S \subset B \times C$ is determined up to vertical isotopy by its \textsl{braid monodromy}, whose definition we now recall.
\begin{definition}
    Let $S \subset B \times C$ be a surface braided over $B$ and let $B_q =\{b_1, \dots, b_n\}\subset \Int B$ be the branch set of the $d$-fold cover $q=\pi_B|_S$. The \textsl{braid monodromy} of $S$ is the representation 
    \[\omega_S \colon \pi_1(B \setminus B_q) \longrightarrow \B_d(C)  \]
    defined as follows. Let $*\in B \setminus B_q$ be a base point for $\pi_1(B \setminus B_q)$ and let $\pi_B|_S^{-1}(*)=\widetilde *=\{*_1, \dots, *_d\}\subset \{*\}\times C$ be the corresponding base point for the configuration space $\Gamma_d(C)$ of $d$-tuples of points in $\{*\}\times C\cong C$. For any $[\gamma]\in \pi_1(B \setminus B_q)$, we set
    \[\omega_S([\gamma])=[\widetilde \gamma]\in \pi_1(\Gamma_d(C))\cong \B_d(C), \]
    where $\widetilde \gamma(t)=q^{-1}(\gamma(t))\subset \{\gamma(t)\}\times C \cong C$ for every $t \in [0,1]$. 
\end{definition}

\section{Lefschetz fibrations}\label{sec:7}

\begin{definition}[Lefschetz fibration]\label{def:lef}
    Let $X$ be a compact connected smooth 4-manifold and let $B$
be a compact connected smooth surface. A \textsl{singular fibration} is a smooth map 
\[f \colon X \longrightarrow B \]
which is a submersion away from finitely many points of $\Int X$ whose images lie in $\Int B$, such that each fiber contains at most one critical point and $f$ is a surface bundle over the complement of the critical image $C_f \subset \Int B$. 

We say that a singular fibration $f\colon X \rightarrow B$ is a \textsl{Lefschetz fibration} if at each critical point $c \in X$ it is locally modeled by the complex non-singular quadratic form \begin{equation}\label{eq:local_model}
    (z_1, z_2) \longmapsto z_1z_2.
\end{equation}
\end{definition}

In the present paper we always assume that the regular fiber of a Lefschetz fibration is a closed connected non-orientable surface. In the oriented setting, one usually requires that the local complex charts yielding (\ref{eq:local_model}) are orientation-preserving. In our setting, this condition does not make sense, since we  work with non-orientable $4$-manifolds.

We now recall the notion of monodromy for a singular fibration $f \colon X \rightarrow B$. In the following, we fix a base point $*\in B \setminus C_f$, which lies in $\Bd B$ if it is non-empty. We set $F=f^{-1}(*)$ as the reference regular fiber for $f$.

\begin{definition}
    Let $f \colon X \rightarrow B$ be a singular fibration with critical image $C_f \subset \Int B$. The \textsl{monodromy of $f$} is defined as the monodromy representation \[\omega_f \colon \pi_1(B \setminus C_f, *) \longrightarrow \M (F)\] of the associated surface bundle \[f|_{X \setminus f^{-1}(C_f)} \colon X \setminus f^{-1}(C_f) \longrightarrow B\setminus C_f.\] 
\end{definition}

\begin{remark}
    If $f \colon X\to B$ is a Lefschetz fibration with regular fiber $F$, then the local model at the critical points as in Definition \ref{def:lef} forces the monodromy of a meridian of each critical value to be a Dehn twist along a simple loop $\gamma\subset F$. Conversely, a singular fibration $f\colon X \to B$ for which the monodromy of a meridian for each critical value is a Dehn twist of the regular fiber can be perturbed to become a Lefschetz fibration. We will use this fact when proving Proposition \ref{thm:piergalliniloi}.
   \end{remark}

\section{Factorizing Lefschetz fibrations through branched covers}\label{sec:lef}

Given a Lefschetz fibration $f \colon X \rightarrow B$, we want to understand when it is possible to find a commutative diagram
\begin{center}
\begin{tikzcd}
X \arrow[r, "p"] \arrow[rd, "f"'] & B\times C \arrow[d, "\pi_B"] \\
                                  & B, 
\end{tikzcd}
\end{center}
where $p$ is a cover branched over a surface braided over $B$. We remark that such factorization can be always found in a local model, in the sense that there exists a diagram
\begin{center}
\begin{tikzcd}
\C^2 \arrow[r, "p"] \arrow[rd, "f"'] & \C^2  \arrow[d, "\pi_{1}"] \\
                                     & \C, 
\end{tikzcd}
\end{center}
where $f(z_1,z_2)=z_1z_2$ and $p(z_1,z_2)=(z_1z_2, z_1+z_2)$. With these choices of $f$ and $p$, the diagram commutes. Moreover, it is easy to check that \[B_p=\{(w_1,w_2)\in \C^2 \mid 4 w_1-w_2^2=0 \}\] and the restriction to $B_p$ of the projection onto the first factor is a $2$-fold simple branched cover of $\C$ with a unique twist point over $0\in \C$. 

The main goal of this section is to prove Theorem \ref{thm:3} via an extension of \cite[Proposition 2]{LoiPiergallini} to the non-orientable setting. In order to do that, we need the following generalization of \cite[Proposition 1]{LoiPiergallini}.

\begin{proposition}\label{thm:piergalliniloi}

Let $B$ and $C$ be connected compact surfaces and let \[p \colon X \longrightarrow B \times C\] be a simple cover branched over a surface $S \subset B \times C$ braided over $B$. Then the composition \[f = \pi_B \circ p \colon X \longrightarrow B\] is a Lefschetz fibration which has the same critical values as $\pi_B|_S$ and one critical point over each twist point of $S$.
\end{proposition}

\begin{remark}
    Unlike previous versions of this result, we do not assume that $X$ is orientable. We are most interested in the case when $X$ is non-orientable. In this situation, $C$ must also be non-orientable.
\end{remark}

Since the proof is very similar to \cite[Proposition 1]{LoiPiergallini}, we provide a sketch for the convenience of the reader.
\begin{proof}[Sketch of proof]
First of all, notice that $f$ is regular at each regular point of $p$. Moreover, one can show that $x\in X$ is a critical point of $f$ if and only if $p(x)\in B \times C$ is a twist point for the braided surface $S$. Let $s_1,\dots, s_n \in S$ be the twist points of $S$ and set $b_i=\pi_B(s_i)$ for every $i=1, \dots, n$. Then $f$ is regular over $B \setminus \{b_1, \dots, b_n\}$. Given $b\in B \setminus \{b_1, \dots, b_n\}$, the regular fiber is given by the pre-image of $\{b\}\times C \subset B \times C$ via $p$. In particular, the restriction of $p$ over $\{b\} \times C$ is a simple cover branched over the transverse intersection between $\{b\}\times C$ and $S$. To summarize, from this we see that the regular fiber $f\inv (b)$ is a surface $F$, where $p|_F$ is a simple cover of $\{b\}\times C$, branched over $S\cap (\{b\}\times C)$. Hence $f$ is a locally trivial fiber bundle away from its critical points. 

To conclude the proof, it is enough to study the behavior of $f$ over each $b_i$. Since $p$ is simple, over each critical value $b_i$ there is only one critical point $x_i\in X$. Let $\ell_i$ be a small loop in $B$ centered around $b_i$. Then $\pi_B\inv(\ell_i)\cong \ell_i\times C$, and $f\inv(\ell_i)=p\inv(\ell_i\times C)$ is a simple cover of $\ell_i\times C$ branched over $S\cap (\ell_i \times C)$, which is a braid $\beta_i\subset \ell_i\times C$. Since $f\inv(\ell_i)$ fibers over $\ell_i\cong S^1$, we see that it's indeed a mapping torus $MT(F,\mu_i)$, and we wish to determine $\mu_i$. To proceed, we first consider what exactly the braid $\beta_i$ is. Recall that $\pi_B|_S$ is a $d$-fold simple branched cover for some $d>0$. The monodromy of $\pi_B|_S$ associated to $b_i$ is a transposition $(i\,j)$ in the symmetric group on $d$ elements. So in $S\cap (\ell_i\times C)$, $\beta_i$ is a half twist along an arc $\alpha_i\subset \{pt\}\times C\subset \ell_i\times C$ between two points in $S\cap (\{pt\}\times C)$, which corresponds to two points in the $i$ and $j$ sheets of $S$ being interchanged. It is a local argument to show that the half twist around $\alpha_i$ lifts to a Dehn twist $t_{\gamma_i}^{\pm 1}$ (cf. Remark \ref{rmk:ht_dt}). Hence $\mu_i = t_{\gamma_i}^{\pm 1}$. Since the monodromy around each critical value of $f$ is a positive or negative Dehn twist, this confirms that $f$ is a Lefschetz fibration.
\end{proof}

\begin{proof}[Proof of Theorem \ref{thm:3}] Let $F$ be the regular fiber of $f$. By assumption, $F$ is a closed connected non-orientable surface. Let $d \geq 6$ be an integer such that $d \equiv \chi(F) \pmod{2}$. Corollary \ref{thm:lift} implies the existence of a simple $d$-fold branched cover \[q \colon F \longrightarrow \mathbb{RP}^2\] such that every Dehn twist around a vanishing cycle of $f$ is the lift of a half twist on $\mathbb{RP}^2$. Moreover, Theorem \ref{thm:1} shows that the lifting homomorphism for $q$ is surjective. One can then use $q$ to build a simple $d$-fold cover $p \colon X \rightarrow B \times \mathbb{RP}^2$ branched over a braided surface $S\subset B \times \mathbb{RP}^2$ such that the critical values of $\pi_{B}|_S$ coincide with those of $f$. By Proposition \ref{thm:piergalliniloi}, we have that $\pi_{B}\circ p$ is a Lefschetz fibration. One can check the monodromy around each of its critical values to see that $\pi_{B}\circ p$ coincides with $f$. The construction of the branched cover $p$ is analogous to the construction of $p\colon X \rightarrow Y \times Z$ in \cite[Proposition 2]{LoiPiergallini}. Here is a summary of how to build $p$.

Let $C_f=\{b_1,\cdots ,b_n\}\subset B$ be the critical image of $f$. For now, assume that the base $B$ is a 2-disk. Let $\{A_i\}$ be a collection of pairwise disjoint disks, where each $A_i$ is a tubular neighborhood of an arc from $\partial B$ to $b_i$. Set $\widehat B= B \setminus (\cup_i \Int A_i)$. Then $f\inv(\widehat B)\cong F\times \widehat B\cong F\times B$. Start with the cover \[p_0=\text{id}_{\widehat B}\times q \colon \widehat B \times F \longrightarrow \widehat B \times \RP^2.\] The branch set for $p_0$ is $\widehat B \times B_q\subset \widehat B \times \RP^2$, which is (trivially) braided over $\widehat B$. For each $A_i$, we construct a cover $q_i\colon W_i\to A_i \times \RP^2$ which we then glue to $f\inv(\widehat B)$ and $\widehat B \times \RP^2$. The construction of $q_i$ is as follows: take the cover \[\text{id}_{A_i}\times q \colon A_i\times F \longrightarrow A_i \times \RP^2.\] Then isotope the vanishing cycle $\gamma_i$ to a curve $\gamma'_i$ which is the lift of a half twist with respect to $q$. Then $q(\gamma'_i)$ is an embedded arc $\alpha_i\subset \RP^2$ between two branch points of $q$. Let $S_i$ be the surface obtained from $A_i \times B_q\subset A_i \times \RP^2$ by adding a half-twisted band along an arc which projects to $\alpha_i$ on the $\RP^2$ factor. Then $\pi_{A_i}|_{S_i}$ is a simple cover with a single branch point, which we may assume is $b_i$. Let $q_i\colon W_i\to A_i \times \RP^2$ be the $d$-fold simple branched cover with branch set $S_i$. After gluing on each $W_i$ and $A_i \times \RP^2$ via $q_i$, it follows from Proposition \ref{thm:piergalliniloi} that the resulting map is a Lefschetz fibration. One can compute and verify that the monodromy coincides with that of $f\colon X\to B$.

Next, assume $B$ is an arbitrary connected compact surface with non-empty boundary. Let $\{\delta_i\}$ be a finite collection of properly embedded arcs on $B$ which are pairwise disjoint, disjoint from the critical values $\{b_1,\cdots , b_n\}$, and such that cutting $B$ along each $\delta_i$ results in a $2$-disk, denoted $B'$. Let $X'$ be the space obtained by cutting $X$ along each $f\inv(\delta_i)$. Then we have a well-defined Lefschetz fibration \[f'\colon X' \longrightarrow B'\] over a disk, which descends to the original Lefschetz fibration $f$ after gluing along each $\delta_i$ and its pre-image. We first apply the procedure in the previous paragraph to $f'$ to obtain a branched cover \[p'\colon X'\longrightarrow B'\times \RP^2\] such that $f' = \pi_{B'}\circ p'$. Then we re-glue $B'$ along the $\delta_i$ arcs, and correspondingly re-glue $X'$ along each $(f')\inv(\delta_i)$. Identifying the regular fiber with $F$, we may assume the gluing of $X'$ along $(f')\inv(\delta_i)$ is given by $\operatorname{id}_{\delta_i}\times \tilde \phi_i$ for some $\tilde \phi_i\in \operatorname{Diff}(F)$. The restriction of the branched cover $p'$ to $(f')\inv(\delta_i)$ is given by \[\operatorname{id}_{\delta_i}\times q\colon \delta_i\times F\to \delta_i\times \mathbb{RP}^2,\] and the branch set along this restriction is $\delta_i\times B_q\subset \delta_i\times \mathbb{RP}^2$. Since the lifting homomorphism $\lambda_q$ is surjective, $\tilde \phi_i$ is, up to isotopy, equal to the lift of some self-diffeomorphism $\phi_i$ of $(\RP^2, B_q)$. By gluing $B'\times \RP^2$ along each $\delta_i\times \RP^2$ via $\operatorname{id}_{\delta_i}\times \phi_i$, we obtain from $p'$ the final branched cover \[p\colon X\longrightarrow B\times \mathbb{RP}^2,\] completing the proof.
\end{proof}

\end{document}